\documentclass[preprint,11pt]{elsarticle}
\usepackage{stmaryrd}
\usepackage{amsmath,amssymb,amsthm}
\usepackage{enumitem}
\usepackage{url}
\usepackage[colorlinks=true,linkcolor=blue,citecolor=blue,urlcolor=blue]{hyperref}

\newcommand{\Q}{\mathbb Q}

\def \Z{\Bbb Z}
\def \C{\Bbb C}

\def \Q{\Bbb Q}

\def \wt{{\textit wt}}

\def \Res{{\rm Res}}
\def \End{{\rm End}}
\def \Aut{{\rm Aut}}

\def \Id{{\rm Id}}
\def \Hom{{\rm Hom}}

\def \im{{\rm Im}}
\def \<{\langle}
\def \>{\rangle}

\def \1{{\bf 1}}

\def \({{\rm (}}
\def \){{\rm )}}

\def \1{{\bf 1}}

\def\Hom{{\rm Hom}}

\def\Res{{\rm Res}}

\def\soc{{\rm soc}}
\newtheorem{Theorem}{Theorem}[section]
\newtheorem{Proposition}[Theorem]{Proposition}
\newtheorem{Lemma}[Theorem]{Lemma}

\newtheorem{Corollary}[Theorem]{Corollary}
\newtheorem{Remark}[Theorem]{Remark}
\newtheorem{Main Theorem}[Theorem]{Main Theorem}

\newtheorem{Definition}[Theorem]{Definition}
\begin{document}

\begin{frontmatter}

\title{A dual theory of twisted Zhu's theory}

\author[add]{Hao Wang\footnote{Email address: whaomath@nwu.edu.cn. Supported by a NSFC grant 12001426}}

\address[add]{School of Mathematics, Northwest University, Xi'an 710127, Shaanxi, China}

\begin{abstract}
The notion of admissible $g$-twisted $V$-comodules is investigated for a graded vertex operator coalgebra $V$ and a finite order automorphism $g\in \Aut V$. We prove that $\mathcal {M}$ is an admissible $g$-twisted $V$-comodule if and only if its graded dual $\mathcal {M}'$ is an admissible $g^{-1}$-twisted $V'$-module for vertex operator algebra $V'$. Then we establish a dual theory of twisted Zhu's theory, i.e., there is a coassociative coalgebra $C_g(V)$ for a graded vertex operator coalgebra $V$, any admissible $g$-twisted $V$-comodule gives a $C_g(V)$-comodule, and vice versa. We also prove that $V$ is $g$-corational, which means every admissible $g$-twisted $V$-comodule is completely reducible, if and only if its dual vertex operator algebra $V'$ is $g$-rational.
\end{abstract}


\end{frontmatter}

\section{Introduction}
In algebraic theory, algebras and coalgebras are a pair of notions which is dual to each other, such as associative algebras and coassociative coalgebras, \cite{DNR,K}, Lie algebras and Lie coalgebras, \cite{M,M1}. Those dual properties also apply to their representations. Vertex operator algebras is another kind of algebras which looks like both associative algebras and Lie algebras, \cite{BPZ,FLM,LL}. In \cite{H,H2}, the author introduces a notion which is called (graded) vertex operator coalgebras. This makes vertex operator algebras and (graded) vertex operator coalgebras to be a pair of notions which is dual to each other. Those dual properties also apply to their representations, see \cite{H1}.

In the representation theory of vertex operator algebras, Zhu's algebras plays a crucial role since it builds a bridge between representations of vertex operator algebra and representations of its Zhu algebra, \cite{DLM1,Z}. In \cite{W}, the author also establishes a dual theory of Zhu's theory, namely, for a graded vertex operator coalgebra $V$, there is a coassociative coalgebra $C(V)$, such that every admissible $V$-comodule gives a $C(V)$-comodule, and vice verse. The author also proves that $V$ has a cosemisimple admissible comodules category if and only if its dual vertex operator algebra $V'$ has a semisimple admissible modules category under certain assumptions.

In vertex operator algebras theory, there is a famous conjecture, called orbiford conjecture. The appearance of twisted admissible modules is the main feature of the orbiford theory, see \cite{DLM4,DRX}. In the study of twisted representations, there is also a twisted Zhu's theory, i.e., for a vertex operator algebra $\mathcal {V}$ and a finite automorphism $g\in\Aut \mathcal {V}$, there is an associative algebra $A_g(\mathcal {V})$, such that every $g$-twisted admissible $V$-module gives an $A_g(\mathcal {V})$-module, and vice verse, see \cite{DLM2,DLM3}. This twisted theory has a lot of generalizations, see \cite{DZz1,DZz2}. From the dualities between vertex operator algebras and (graded) vertex operator coalgebras and their admissible representations, we also hope to establish a twisted theory for (graded) vertex operator coalgebras which is the dual theory of twisted Zhu's theory for vertex operator algebras. This is our motivation.

Let $V$ be a graded vertex operator coalgebra and $g\in\Aut V$ a finite order automorphism. To establish the dual theory of twisted Zhu's theory, the notion of admissible $g$-twisted $V$-comodules is necessary. For this purpose, we first introduce the notion of $g$-twisted $V$-comodules. Then, we prove that $g^{-1}$ is an automorphism of the dual vertex operator algebra $V'$, and $\mathcal {M}$ is an admissible $g$-twisted $V$-comodule if and only if the graded dual space $\mathcal {M}'$ is an admissible $g^{-1}$-twisted $V'$-module and vice verse.

Let $\mathcal {V}$ be a vertex operator algebra and $g\in\Aut\mathcal {V}$ a finite order automorphism, the associative algebra $A_g(\mathcal {V})$ is defined to a quotient space of $\mathcal {V}$. Using the theorems of homomorphisms, $A_g(\mathcal {V})$ is actually a quotient space of the fixed points vertex operator sub algebra $\mathcal {V}^g$, \cite{DLM2,DLM3}. When we take the dual spaces, quotient spaces become sub spaces and vice verse. Hence for a (graded) vertex operator coalgebra $V$ and a finite order automorphism $g\in\Aut V$, we define the coassociative coalgebra $C_g(V)$ to be a sub space of $V$, then we also prove that $C_g(V)$ is actually a sub space of the fixed quotient vertex operator coalgebra $V^g$. It should be noticed that the fixed points subspace is not a (graded) vertex operator coalgebra any more. Then we define a well-defined coproduct $\Delta_g:C_g(V)\rightarrow C_g(V)\otimes C_g(V)$, and prove that $\Delta_g$ is coassociative. This proves $C_g(V)$ is a coassociative coalgebra under certain assumptions. After getting the coassociative coalgebra $C_g(V)$, we prove that every admissible $g$-twisted $V$-comodule gives a $C_g(V)$-comodule. Hence we get a functor $\Omega_g$ from admissible $g$-twisted $V$-comodules category to $C_g(V)$-comodules category. This completes the first half of the twisted theory for (graded) vertex operator coalgebras.

Conversely, we also hope to get an admissible $g$-twisted $V$-comodule $\mathcal {L}_g(M)$ from a given $C_g(V)$-comodule $M$. We use the method given in \cite{W}. First we work on the relations between $C_g(V)$ and the associative algebra $A_{g^{-1}}(V')$ associated to $V'$. Under certain assumptions, we prove that there is an isomorphism $\Phi_g:A_{g^{-1}}(V')\rightarrow C_g(V)^*.$ Hence for a $C_g(V)$-comodule $M$, $M^*$ is a $C_g(V)^*$-module, thus an $A_{g^{-1}}(V')$-module. Now using Dong-Li-Mason's theory, there is an admissible $g^{-1}$-twisted $V'$-module $L_{g^{-1}}(M^*)$ such that $L_{g^{-1}}(M^*)_0=M^*,$ see \cite{DLM2}. Then using the dual relations between admissible $g$-twisted $V$-comodules and admissible $g^{-1}$-twisted $V'$-modules, we arrive at an admissible $g$-twisted $V$-comodule $L_{g^{-1}}(M^*)'.$ In general, this admissible $g$-twisted $V$-comodule is too huge, and it is not what we want. We define $\mathcal {L}_g(M)$ to be a sub quotient of $L_{g^{-1}}(M^*)'$, and prove that $\mathcal {L}_g(M)_0=M$. Furthermore, if $\dim M$ is countable, then each homogeneous subspace of $\mathcal {L}_g(M)$ has countable dimension. Hence we get a functor from $C_g(V)$-comodules category to admissible $g$-twisted $V$-comodules category, such that $\Omega_g\circ\mathcal {L}_g=\Id$, and $\mathcal {L}_g$ also sends simple objects to simple objects. This completes the other half of the twisted theory for (graded) vertex operator coalgebras.

Using this twisted theory, we prove that if $V$ is $g$-corational, i.e., every admissible $g$-twisted $V$-comodule is completely reducible, then $C_g(V)$ is a cosemisimple coassociative coalgebra under certain assumptions.

Finally, we also want to study the relations between $g$-corationality of $V$ and $g$-rationality of $V'$. For this purpose, we need to introduce the higher level twisted coassociative coalgebras and establish a higher level twisted theory for (graded) vertex operator coalgebras which is the dual theory of higher level twisted Zhu's theory for vertex operator algebras. After doing that, we prove that $V$ is $g$-corational if and only if all $C_g^k(V)$ are cosemisimple coassociative coalgebras for all $k\in\frac{1}{T}\mathbb{N}$, where $T$ is the order of $g$. Then using a result of \cite{DLM3}, we prove that $V$ is $g$-corational if and only if $V'$ is $g^{-1}$-rational. Again using a result of \cite{DLM3} and \cite{DJ1}, we can prove that $V$ is $g$-corational if and only if $V'$ is $g$-rational. This is a generalization of a classical result, i.e., $A$ is a semisimple associative algebra if and only if $A^*$ is a cosemisimple coassociative coalgebra.

This paper is organised as follows: In section 2, we recall several definitions and properties about associative algebras and coassociative coalgebras, their modules and comodules briefly. We also recall definitions and properties of vertex operator algebras and vertex operator coalgebras, admissible $g$-twisted modules. In section 3, we first introduce the notion of admissible $g$-twisted $V$-comodules for a graded vertex operator coalgebra $V$ and a finite order automorphism $g\in\Aut V$. Then we prove that $\mathcal {M}$ is an admissible $g$-twisted $V$-comodule if and only if $\mathcal {M}'$ is an admissible $g^{-1}$-twisted $V'$-module and vice verse. Section 4 is about the coassociative coalgebra $C_g(V)$. Section 5 and 6 are the constructions of functors $\Omega_g$ and $\mathcal {L}_g$. We prove that if $V$ is $g$-corational, then $C_g(V)$ is a cosemisimple coassociative coalgebra under certain assumptions in section 6. In section 7, we introduce the higher level coassociative coalgebras, and prove that $V$ is $g$-corational if and only if $V'$ is $g$-rational.

\section{Preliminaries}

In this paper, we work over complex field $\mathbb{C},$ all vector spaces, linear maps will be over $\mathbb{C},$ $\otimes$ means $\otimes _\mathbb{C}$. $\Res_zf(z)$ is the coefficient of $z^{-1}$. Furthermore, we have $$\Res_z\frac{d}{dz}f(z)\cdot g(z)=-\Res_zf(z)\cdot\frac{d}{dz}g(z).$$
For any vector spaces $V,W$, the flipping map $\tau:V\otimes W\rightarrow W\otimes V$ is defined by $$\tau(v\otimes w)=w\otimes v,$$ for any $v\in V,w\in W.$ In this paper, we always view $V\otimes \mathbb{C}=\mathbb{C}\otimes V=V$ naturally. For $n\in\mathbb{C}$, $(z_1+z_2)^n=\sum_{i\geq0}\tbinom{n}{i}z_1^{n-i}z_2^i.$ Let $V$ be a vector space, denote $$V^*=\Hom(V,\mathbb{C})$$ by the dual space of $V$. Let $(\cdot,\cdot)$ be the natural pair between $V^*$ and $V$, and $$V\{z\}=\{\sum_{k\in\mathbb{C}}v_kz^k|v_k\in V\}.$$

\subsection{Algebras and coalgebras}
In this subsection, we recall several concepts and results about algebras and coalgebras.
\begin{Definition}\cite{DNR}\label{DA2.1}
  An associative algebra is a triple $(A,\mu,\eta)$, where $A$ is a vector space, $\mu: A\otimes A\rightarrow A$ and $\eta:\mathbb{C}\rightarrow A$ are linear maps such that:

  (i) For any $a\in A$, we have $\mu(a\otimes\eta(1))=\mu(\eta(1)\otimes a)=a$. There are called unit identities.

  (ii) For any $a,b,c\in A$, the following identity $$\mu(a\otimes\mu(b\otimes c))=\mu(\mu(a\otimes b)\otimes c)$$ holds. This identity is called associativity. We write $\mu(a\otimes b)=ab$.
\end{Definition}

\begin{Definition}\cite{DNR}\label{DA2.2}
  Let $(A,\mu,\eta)$ be an algebra and $B\subseteq A$ a subspace of $A$. If $(B,\mu,\eta)$ is also an algebra, we say it is a sub algebra of $A$.

  If $\mu(A\otimes B)\subseteq B$, we say $B$ is a left ideal of $A$.

  If $\mu(B\otimes A)\subseteq B$, we say $B$ is a right ideal of $A$.

  $B$ is an ideal if it is both a left ideal and right ideal.

  $A$ is called simple if there is no nontrivial ideal. $A$ is called semisimple if $A$ is a direct sum of simple ideals.
\end{Definition}

\begin{Definition}\cite{DNR}
Let $(A,\mu_A,\eta_A)$ and $(B,\mu_B,\eta_B)$ be two algebras. A linear map $\phi:A\rightarrow B$ is called a homomorphism if $$\phi\circ\eta_A=\eta_B,\phi\circ\mu_A=\mu_B\circ\phi\otimes\phi.$$

  Similarly, we have the definitions of monomorphism, epimorphism and isomorphism, etc.
\end{Definition}

\begin{Definition}\cite{DNR}\label{DAM2.1}
  Let $(A,\mu,\eta)$ be an algebra, $N$ a vector space and $\varrho:A\otimes N\rightarrow N$ a linear map. $(N,\varrho)$ is called a left module of $A$ if

  (i) For any $n\in N$, we have $\varrho(1\otimes n)=n$. This is called left unit identity.

  (ii) For any $a,b\in A,n\in N$, we have $\varrho(a\otimes\varrho(b\otimes n))=\varrho(\mu(a\otimes b)\otimes n)$. This is called associative identity. We write $\varrho(a\otimes n)=an.$
\end{Definition}

\begin{Definition}\cite{DNR}\label{DAM2.2}
  Let $(A,\mu,\eta)$ be an algebra and $(N,\varrho_N)$ a left module of $A$. Let $N^1$ be a sub space of $N$. If $(N^1,\varrho_N|_{N^1})$ is also a left module of $A$, we say $N^1$ is a sub module of $N$.

  $N$ is called simple if there is no nontrivial sub module. $N$ is called semisimple if $N$ is a direct sum of simple sub modules.

  Let $(N^1,\varrho_{N^1})$ and $(N^2,\varrho_{N^2})$ be two left modules of $A$. A linear map $\psi:N^1\rightarrow N^2$ is called an $A$-module homomorphism if $$\psi\circ\varrho_{N^1}=\varrho_{N^2}\circ(\Id\otimes \psi).$$

  Similarly, we have the definitions of monomorphism, epimorphism and isomorphism, etc.
\end{Definition}

\begin{Definition}\cite{DNR}\label{DC2.1}
  A coassociative coalgebra is a triple $(C,\Delta,\epsilon)$, where $C$ is a vector space, $\Delta:C\rightarrow C\otimes C$ and $\epsilon:C\rightarrow \mathbb{C}$ are linear maps such that:

  (i) For any $a\in C$, the following identities $$(\Id\otimes \epsilon)\circ \Delta(a)=a=(\epsilon\otimes \Id)\circ \Delta(a)$$ hold. There are called counit identities.

  (ii) For any $a\in C$, the following identity $$(\Id\otimes \Delta)\circ \Delta(a)=(\Delta\otimes \Id)\circ \Delta(a)$$ holds. This identity is called coassociativity.
\end{Definition}

\begin{Remark}
By coassociativity, we will write $\Delta(a)=\sum a'\otimes a''$ for simplicity, and this will cause no confusion.
\end{Remark}

\begin{Definition}\cite{DNR}\label{DC2.2}
  Let $(C,\Delta,\epsilon)$ be a coalgebra and $D\subseteq C$ a subspace of $C$. If $\Delta(D)\subseteq D\otimes D$, we say $(D,\Delta,\epsilon)$ is a sub coalgebra of $C$.

  If $\Delta(D)\subseteq C\otimes D$, we say $(D,\Delta,\epsilon)$ is a left coideal of $C$.

  If $\Delta(D)\subseteq D\otimes C$, we say $(D,\Delta,\epsilon)$ is a right coideal of $C$.

  If $\Delta(D)\subseteq C\otimes D+D\otimes C$ and $\epsilon(D)=0$, we say $(D,\Delta,\epsilon)$ is a coideal of $C$.

  $C$ is called simple if there is no nontrivial sub coalgebra. $C$ is called cosemisimple if $C$ is a direct sum of simple sub coalgebras.
\end{Definition}

\begin{Definition}\cite{DNR}\label{DCM2.1}
  Let $(C,\Delta,\epsilon)$ be a coalgebra, $M$ a vector space and $\Delta_M:M\rightarrow C\otimes M$ a linear map. $(M,\Delta_M)$ is called a left comodule of $C$ if

  (i) For any $m\in M$, the following identity $$(\epsilon\otimes \Id)\circ \Delta_M(m)=m$$ holds. This is called left counit identity.

  (ii) For any $m\in M$, the following identity $$(\Id\otimes \Delta_M)\circ \Delta_M(m)=(\Delta\otimes \Id)\circ \Delta_M(m)$$ holds. This is called coassociative identity.
\end{Definition}

\begin{Definition}\cite{DNR}
Let $(C,\Delta_C,\epsilon_C)$ and $(D,\Delta_D,\epsilon_D)$ be two coalgebras. A linear map $\phi:C\rightarrow D$ is called a homomorphism if $$\epsilon_D\circ\phi=\epsilon_C,\Delta_D\circ\phi=\phi\otimes\phi\circ\Delta_C.$$

  Similarly, we have the definitions of monomorphism, epimorphism and isomorphism, etc.
\end{Definition}

\begin{Remark}
By coassociativity, we will write $\Delta_M(m)=\sum m'\otimes m''$ for simplicity, and this will cause no confusion. In this notation, we have $m'\in C,m''\in M.$
\end{Remark}

\begin{Definition}\cite{DNR}\label{DCM2.2}
  Let $(C,\Delta,\epsilon)$ be a coalgebra and $(M,\Delta_M)$ a left comodule of $C$. Let $M^1$ be a sub space of $M$. If $(M^1,\Delta_M|_{M^1})$ is also a left comodule of $C$, we say $M^1$ is a sub comodule of $M$.

  $M$ is called simple if there is no nontrivial sub comodule. $M$ is called cosemisimple if $M$ is a direct sum of simple sub comodules.

  Let $(M^1,\Delta_{M^1})$ and $(M^2,\Delta_{M^2})$ be two left comodules of $C$. A linear map $\psi:M^1\rightarrow M^2$ is called a $C$-comodule homomorphism if $$(\Id\otimes \psi)\circ \Delta_{M^1}=\Delta_{M^2}\circ \psi.$$

  Similarly, we have the definitions of monomorphism, epimorphism and isomorphism, etc.
\end{Definition}

\begin{Lemma}\cite{DNR}\label{coss}
$C$ is cosemisimple if and only if every $C$-comodule is cosemisimple.
\end{Lemma}

\begin{Proposition}\cite{DNR,K}\label{A-C}
Let $(A,\mu,\eta)$ be an algebra and $(C,\Delta,\epsilon)$ a coalgebra. Let $(M,\Delta_M)$ be a left comodule of $C$. Then, we have

(i) $C^*$ is an algebra with $$(\mu_{C^*}(f\otimes g),a)=\sum f(a')g(a''),\eta_{C^*}(1)=\epsilon,$$ where $f,g\in C^*, a\in C$.

(ii) $A^\circ$ is a coalgebra with $$(\Delta_{A^\circ}(f),a\otimes b)=f(ab),\epsilon_{A^\circ}(f)=f(\eta(1)),$$ where $f\in A^\circ, a,b\in A.$ If $\dim A<\infty,$ then $A^\circ=A^*.$

(iii) $M^*$ is a left $C^*$-module with $(\varrho_{M^*}(f\otimes m^*))m=\sum f(m')m^*(m'')$, where $f\in C^*,m^*\in M^*,m\in M.$
\end{Proposition}

\subsection{Vertex operator algebras and coalgebras}
In this subsection, we recall several concepts and results about vertex operator algebras and coalgebras.
 \begin{Definition}\cite{LL}
  Let $\mathcal {V}=\oplus_{s\in \mathbb{Z}}\mathcal {V}_s$ be a $\mathbb{Z}$-graded vector space with $\mathcal {V}_s=0$, $0\gg s$, $\dim \mathcal {V}_s<\infty,\forall s\in \mathbb{Z}$, and $\textbf{1}\in \mathcal {V}_0,\omega\in \mathcal {V}_2,Y(\cdot,z):\mathcal {V}\otimes \mathcal {V}\rightarrow \End(\mathcal {V})[[z,z^{-1}]],u\otimes v\mapsto Y(u,z)v=\sum_{t\in\mathbb{Z}}u_tvz^{-t-1}$, where $u_t\in \End(\mathcal {V})$. Then, $(\mathcal {V},Y,\textbf{1},\omega)$ is called a vertex operator algebra if the following hold:

  (i) For $ u,v\in \mathcal {V},$ $u_tv=0$, if $t\gg0$.

  (ii) $Y(\textbf{1},z)v=v,\lim_{z\rightarrow0}Y(v,z)\textbf{1}=v$, for $ v\in \mathcal {V}$.

  (iii) Write $Y(\omega,z)=\sum_{t\in\mathbb{Z}}\omega_tz^{-t-1}=\sum_{t\in\mathbb{Z}}L(t)z^{-t-2}$, then
  \begin{eqnarray*}
  &&\mathcal {V}_s=\{v\in \mathcal {V}|L(0)v=sv\},~Y(L(-1)v,z)=\frac{d}{dz}Y(v,z),\\
  &&[L(p),L(q)]=(p-q)L(p+q)+\delta_{p+q,0}\frac{p^3-p}{12}d,
  \end{eqnarray*}
  where $d\in\mathbb{C}$ is called central charge of $\mathcal {V}$. For $v\in \mathcal {V}_s$, $v$ is said to be homogeneous and the weight $\wt v$ of $v$ is defined to be $s$.

  (iv) For $ u,v\in \mathcal {V}$, we have
  \begin{eqnarray*}
  & &z_0^{-1}\delta(\frac{z_1-z_2}{z_0})Y(u,z_1)Y(v,z_2)-z_0^{-1}\delta(\frac{-z_2+z_1}{z_0})Y(v,z_2)Y(u,z_1)\\
  & &\ \ \ \ \ =z_1^{-1}\delta(\frac{z_2+z_0}{z_1})Y(Y(u,z_0)v,z_2).
  \end{eqnarray*}
  \end{Definition}

\begin{Definition}\cite{LL}
Let $(\mathcal {V},Y(\cdot,z),\textbf{1},\omega)$ be a vertex operator algebra. A bijection $g\in\End \mathcal {V}$ is called an automorphism if $g(\textbf{1})=\textbf{1},g(\omega)=\omega,g(Y(u,z)v)=Y(g(u),z)g(v)$ for all $u.v\in\mathcal {V}$. Let $\Aut \mathcal {V}$ be the set of all automorphisms of $V$.

$g\in\Aut\mathcal {V}$ is called a finite order automorphism if there is a positive integer $r\in\mathbb{N}$ such that $g^r=\Id,$ the smallest positive integer is called the order of $g$, denoted by $o(g)$.

A linear map $\phi:\mathcal {V}\rightarrow \mathcal {V}$ is called a derivation if $$\phi(\textbf{1})=\phi(\omega)=0,\phi(Y(u,z)v)=Y(\phi(u),z)v+Y(u,z)\phi(v),$$ for any $u,v\in\mathcal {V}$.
\end{Definition}

\begin{Definition}\cite{DLM2}
Let $(\mathcal {V},Y(\cdot,z),\textbf{1},\omega)$ be a vertex operator algebra, $g\in\Aut\mathcal {V}$ an automorphism of order $T$. Set $\mathcal {V}^{(r)}=\{v\in\mathcal {V}|g(v)=e^{2\pi i\frac{r}{T}}\}$. If $\mathcal {N}$ is a vector space equipped
with a linear map
\begin{align*}
Y_{\mathcal {N}}:\mathcal {V}&\to (\End \mathcal {N})[[z, z^{-1}]],\\
v&\mapsto Y_{N}(v,z)=\sum_{t\in\Q}v_tz^{-t-1},\,v_t\in \End \mathcal {N},
\end{align*}
satisfying the following conditions:

(i) For any $u\in \mathcal {V}^{(r)},\ n\in \mathcal {N}$,
\begin{align*}
&\ \ \ \ \ \ \ \ \ \ \ \ \ \ \ \ \ Y_\mathcal {N}(u,z)=\sum_{t\in\Z+\frac{r}{T}}v_tz^{-t-1};\\
&\ \ \ \ \ \ \ \ \ \ \ \ \ \ \ \ \ u_tn=0 \text{ for } t\gg0;\\
&\ \ \ \ \ \ \ \ \ \ \ \ \ \ \ \ \ \  Y_\mathcal {N}(\1, z)=\Id_\mathcal {N}.\\
\end{align*}

(ii) For any $u\in V^{(r)}, v\in V$,
\begin{eqnarray}
&&z_0^{-1}\delta(\frac{z_1-z_2}{z_0})Y_\mathcal {N}(u,z_1)Y_\mathcal {N}(v,z_2)-z_0^{-1}\delta(\frac{-z_2+z_1}{z_0})Y_\mathcal {N}(v,z_2)Y_\mathcal {N}(u,z_1)\nonumber\\
&&\ \ \ \ \ =z_2^{-1}\delta(\frac{z_1-z_0}{z_2})(\frac{z_1-z_0}{z_2})^{\frac{-r}{T}}Y_\mathcal {N}(Y(u,z_0)v,z_2).\label{ID2.3}
\end{eqnarray}
Then, $(\mathcal {N},Y_\mathcal {N}(\cdot,z))$ is called a weak $g$-twisted $\mathcal {V}$-module.

If $g=1$, this reduces to the definition of weak $\mathcal {V}$-module.

A weak $g$-twisted $\mathcal {V}$-module $\mathcal {N}$ is called an admissible $g$-twisted $\mathcal {V}$-module if $\mathcal {N}$ has a $\frac{1}{T}\mathbb{N}$-gradation $\mathcal {N}=\bigoplus_{t\in\frac{1}{T}\mathbb{N}}N_t$ such that
\begin{align*}\label{AD1}
a_sN_t\subset N_{\wt{a}+t-s-1}
\end{align*}
for any homogeneous $a\in \mathcal {V}$ and $s\in\frac{1}{T}\mathbb{Z},\,t\in\frac{1}{T}\mathbb{N}$.

An admissible $g$-twisted $\mathcal {V}$-module $\mathcal {N}$ is said to be irreducible if $\mathcal {N}$ has no non-trivial admissible $g$-twisted $\mathcal {V}$-submodule. When an admissible $g$-twisted $\mathcal {V}$-module $\mathcal {N}$ is a direct sum of irreducible admissible $g$-twisted submodules, $\mathcal {N}$ is called completely reducible.

$\mathcal {V}$ is called $g$-rational if all admissible $g$-twisted $\mathcal {V}$-modules are completely reducible.
\end{Definition}

\begin{Definition}\cite{H,H2}\label{Def2.17}
  A vertex coalgebra is a triple $(V,\Yup (z),c)$, where $V$ is a vector space, $c: V\rightarrow\mathbb{C}$ is a linear map, $\Yup (z):V\rightarrow V\otimes V[[z,z^{-1}]]$, $v\mapsto \sum_{k\in\mathbb{Z}}\Delta_k(v)z^{-k-1}$ is a linear map, which satisfy following conditions:

  (i) For any $v\in V$, we have $\Delta_k(v)=0$ for $k<<0.$

  (ii) For any $v\in V$, we have
    \begin{eqnarray}\label{LU3.1}
      (c\otimes \Id)\circ\Yup (z)v=v,
    \end{eqnarray} and
    \begin{eqnarray}\label{CI3.2}
      (\Id\otimes c)\circ\Yup (z)v\in V[[z]], \label{CI3.2}\\
      \lim_{z\rightarrow 0}(\Id\otimes c)\circ\Yup (z)v=v.\label{CI3.3}
    \end{eqnarray}

  (iii) The following identity
    \begin{eqnarray}\label{JI3.4}
      &&z_0^{-1}\delta(\frac{z_1-z_2}{z_0})(\Id\otimes \Yup (z_2))\circ \Yup (z_1)-z_0^{-1}\delta(\frac{z_2-z_1}{-z_0})(\tau\otimes\Id)\circ(\Id\otimes \Yup (z_1))\circ \Yup (z_2)\nonumber\\
      &&\ \ \ \ \ =z_1^{-1}\delta(\frac{z_2+z_0}{z_1})(\Yup (z_0)\otimes \Id)\circ \Yup (z_2)
    \end{eqnarray}
  holds on $V$. This identity is called Jacobi identity.

  $A$ vertex coalgebra $(V,\Yup(z),c)$ is called graded if there is a $\mathbb{Z}$-gradation $V=\oplus_{s\in\mathbb{Z}}V_s$ on $V$ such that each homogeneous space is finite dimensional and $V_s=0$ for $s<<0.$ View $V\otimes V$ as a $\mathbb{Z}$-graded vector space with natural gradation, i.e., $$(V\otimes V)_t=\oplus_{s\in\mathbb{Z}}V_{t-s}\otimes V_{s}.$$ $\Delta_k$ is a homogeneous linear map of degree $k+1$, i.e., for any $v\in V_s$, we have $$\Delta_k(v)\in (V\otimes V)_{s+k+1}.$$

  A $\mathbb{Z}$-graded vertex coalgebra is called a $\mathbb{Z}$-graded vertex operator coalgebra, if there is another linear map $\rho:V\rightarrow \mathbb{C}$, and write $(\rho\otimes \Id)\circ\Yup (z)=\sum_{k\in\mathbb{Z}}L(k)z^{k-2}$, then $L(k)\in\End(V)$ and the following Virasoro identity
    \begin{eqnarray}\label{VI3.5}
     [L(k),L(j)]=(k-j)L(k+j)+\frac{k^3-k}{12}\delta_{k+j,0}d
    \end{eqnarray}
  holds on $V$. $d$ is called the rank of $V$. Furthermore, $V_s$ is an eigenspace of $L(0)$ with eigenvalue $s$ and the following identity
    \begin{eqnarray}\label{L(1D)3.6}
     (L(1)\otimes\Id)\circ \Yup (z)=\frac{d}{dz}\Yup (z)
    \end{eqnarray}
  holds on $V$. For any $v\in V_s$, we say $v$ is homogeneous of weight $s$, denoted by $s=\wt v.$

  If there is no confusion, we may say $V$ is a graded vertex operator coalgebra for simplicity.
\end{Definition}

\begin{Definition}
Let $(V,\Yup(z),c,\rho)$ be a graded vertex operator algebra and $U\subseteq V$ a graded subspace of $V$. $U$ is called a coideal of $V$ if for any $u\in U$, $$c(u)=\rho(u)=0,\Delta_k(u)\in V\otimes U+U\otimes V,k\in\mathbb{Z}.$$
\end{Definition}

\begin{Proposition}\label{Quo}
Let $(V,\Yup(z),c,\rho)$ be a graded vertex operator algebra and $U\subseteq V$ a coideal of $V$. Then $V/U$ is a graded vertex operator algebra.
\end{Proposition}
\proof The proof of this Proposition is routine. First, $$\Yup_{V/U}(z)\circ\pi=\pi\otimes\pi\circ\Yup(z),$$ where $\pi:V\rightarrow V/U$ is the canonical projection. Since $U$ is a coideal, we can see that $\Yup_{V/U}(z)$ is well-defined, and there are unique linear maps $c_{V/U},\rho_{V/U}$ such that $$c_{V/U}\circ\pi=c,\rho_{V/U}\circ\pi=\rho.$$ Now the other axioms are obvious.

Furthermore, $\pi$ is a homomorphism of graded vertex operator coalgebras.                      $\hfill\Box$

\begin{Proposition}\cite{H2}
In the definition of vertex coalgebras, Jacobi identity (\ref{JI3.4}) is equivalent to following two conditions:

(i) (Weak cocommutativity) For $v\in V$, there exist $q\in \mathbb{N}$ such that $$(z_1-z_2)^q(\Id\otimes \Yup (z_2))\circ \Yup (z_1)v=(z_1-z_2)^q(\tau\otimes\Id)\circ(\Id\otimes \Yup (z_1))\circ \Yup (z_2)v.$$

(ii) (Weak coassociativity) For $v\in V$, there exist $q\in \mathbb{N}$ such that $$(z_0+z_2)^q(\Yup (z_0)\otimes \Id)\circ \Yup (z_2)v=(z_0+z_2)^q(\Id\otimes \Yup (z_2))\circ \Yup (z_0+z_2)v.$$
\end{Proposition}

\begin{Proposition}\cite{H,H1,H2}\label{VC-VA}
Let $(V,\Yup (z),c,\rho)$ be a graded vertex operator coalgebra and $(\mathcal {V},Y,\textbf{1},\omega)$ a vertex operator algebra. Suppose $V_0\nsubseteq\ker c,V_2\nsubseteq\ker\rho.$ Then we have

(i) $(\mathcal {V}'=\oplus_{s\in\mathbb{Z}}\mathcal {V}_s^*,\Yup_{\mathcal {V}'}(z),c_{\mathcal {V}'},\rho_{\mathcal {V}'})$ is a graded vertex operator coalgebra with
\begin{eqnarray*}
&&c_{\mathcal {V}'}=\textbf{1}^*,\rho_{\mathcal {V}'}=\omega^*,\\
&&(\Yup_{\mathcal {V}'}(z)v^*,u\otimes v)=(v^*,Y(u,z)v),
\end{eqnarray*}
where $v^*\in \mathcal {V}',u,v\in\mathcal {V}$, $\textbf{1}^*,\omega^*$ are the dual elements of $\textbf{1},\omega.$

(ii) $(V'=\oplus_{s\in\mathbb{Z}}V_s',Y_{V'}(\cdot,z),\textbf{1}_{V'},\omega_{V'})$ is a vertex operator algebra with
    \begin{eqnarray*}
     &&\textbf{1}_{V'}=c|_{V_0},\omega_{V'}=\rho|_{V_2},\\
     &&(Y(u^*,z)v^*,v)=(u^*\otimes v^*,\Yup(z)v),
    \end{eqnarray*}
where $u^*,v^*\in V',v\in V.$
\end{Proposition}

\section{Twisted admissible $V$-comodules}
In this section, we introduce the notion of twisted admissible $V$-comodules for a graded vertex operator coalgebra $V$, and prove the restricted dual is a twisted admissible $V'$-module and vice verse.

\begin{Definition}
Let $(V,\Yup_{V} (z),c_{V},\rho_{V})$ and $(V^1,\Yup_{V^1} (z),c_{V^1},\rho_{V^1})$ be two graded vertex operator coalgebras. A linear map $\phi:V\rightarrow V^1$ is called a homomorphism if $$c_{V^1}=c_{V}\circ\phi,\rho_{V^1}=\rho_{V}\circ\phi,\Yup_{V^1}(z)\circ\phi=\phi\otimes\phi\circ\Yup_{V}(z).$$  Similarly, we have the definitions of monomorphism, epimorphism and isomorphism, etc.

If $V=V^1$, $\phi$ is called an endomorphism. Furthermore, if $\phi$ is isomorphic, it is called an automorphism. Let $\Aut V$ denote the set of all automorphisms of $V.$

$g\in\Aut V$ is called a finite order automorphism if there is a positive integer $r\in\mathbb{N}$ such that $g^r=\Id,$ the smallest positive integer is called the order of $g$, denoted by $o(g)$.

A linear map $\phi:V\rightarrow V$ is called a derivation if $$c_{V}\circ\phi=\rho_{V}\circ\phi=0,\Yup_{V^1}(z)\circ\phi=\phi\otimes\Id\circ\Yup_{V}(z)+\Id\otimes\phi\circ\Yup_{V}(z).$$
\end{Definition}

\begin{Proposition}\label{Prop3.2}
Let $(V,\Yup(z),c,\rho)$ be a graded vertex operator coalgebra.

(i) $g$ is an automorphism of $V$ if and only if $g$ is an automorphism of $V'$.

(ii) $\phi$ is a derivation of $V$ if and only if $\phi$ is a derivation of $V'$.
\end{Proposition}
\proof (i) Suppose $g$ is an automorphism of $V$, then $g^{-1}$ is also an automorphism of $V$. For any $v^*\in V'$, define $$(g(v^*),v)=(v^*,g^{-1}(v)),$$ where $v\in V.$ Now by Proposition \ref{VC-VA}, we have
\begin{eqnarray*}
&&(g(c|_{V_0}),v)=(c_{V_0},g^{-1}(v))=(c_{V_0}\circ g^{-1},v)=(c|_{V_0},v),\\
&&(g(\rho|_{V_0}),v)=(\rho_{V_0},g^{-1}(v))=(\rho_{V_0}\circ g^{-1},v)=(\rho|_{V_0},v),
\end{eqnarray*}
and
\begin{eqnarray*}
&&(g(Y(v^*,z)u^*,v)=(Y(v^*,z)u^*,g^{-1}(v))=(v^*\otimes u^*,\Yup(z)g^{-1}(v))\\
&=&(v^*\otimes u^*,g^{-1}\otimes g^{-1}\circ\Yup(z)(v))=(Y(g(v^*),z) g(u^*),v),
\end{eqnarray*}
where $v^*,u^*\in V',v\in V$. Hence, $g$ is an automorphism of $V'$.

Conversely, it is similar to prove if $g$ is an automorphism of $V'$, it is also an automorphism of $V$.

(ii) Let $\phi$ be a derivation of $V$. For any $v^*\in V'$, define $(\phi(v^*),v)=(v^*,-\phi(v)),$ where $v\in V$. Now by Proposition \ref{VC-VA}, we have
\begin{eqnarray*}
&&(\phi(c|_{V_0}),v)=(c_{V_0},-\phi(v))=-(c_{V_0}\circ \phi,v)=0,\\
&&(\phi(\rho|_{V_0}),v)=(\rho_{V_0},-\phi(v))=-(\rho_{V_0}\circ \phi,v)=0,
\end{eqnarray*}
and
\begin{eqnarray*}
&&(\phi(Y(v^*,z)u^*,v)=(Y(v^*,z)u^*,-\phi(v))=-(v^*\otimes u^*,\Yup(z)\phi(v))\\
&=&-(v^*\otimes u^*,\phi\otimes\Id\circ\Yup_{V}(z)+\Id\otimes\phi\circ\Yup_{V}(z)(v))\\
&=&(Y(\phi(v^*),z)u^*+Y(v^*,z)\phi(u^*),v),
\end{eqnarray*}
where $v^*,u^*\in V',v\in V$. Hence, $\phi$ is a derivation of $V'$.

Conversely, it is similar to prove if $\phi$ is a derivation of $V'$, then it is also a derivation of $V$.                                                $\hfill\Box$

\begin{Definition}\label{Def3.3}
Let $(V,\Yup(z),c,\rho)$ be a graded vertex operator coalgebra and $g\in \Aut V$ a finite order automorphism of order $T$. Define $$V^g=V/U,$$where $U=\oplus_{i=1}^{T-1}V^{(r)}$ and $V^{(r)}=\{v\in V|g(v)=e^{\frac{2\pi ir}{T}}v\}.$ $V^g$ is called the fixed points quotient of $V$.
\end{Definition}

\begin{Proposition}\label{Prop3.4}
Let $(V,\Yup(z),c,\rho)$ be a graded vertex operator coalgebra and $g\in \Aut V$. Suppose $o(g)=T$, then $(V^g,\Yup_{V^g}(z),c\circ\pi,\rho\circ\pi)$ is also a graded vertex operator coalgebra, $V^g$ is called the fixed points quotient graded vertex operator coalgebra of $V.$
\end{Proposition}
\proof By Proposition \ref{Quo}, we prove $U$ is a coideal. First, we have
\begin{eqnarray*}
L(k)\circ g&=&\Res_zz^{-k+1}(\rho\otimes \Id)\circ\Yup(z)\circ g\\
&=&\Res_zz^{-k+1}(\rho\otimes \Id)\circ(g\otimes g)\circ\Yup(z)\\
&=&\Res_zz^{-k+1}(\rho\circ g\otimes g)\circ\Yup(z)\\
&=&\Res_zz^{-k+1}g\circ(\rho\otimes \Id)\circ\Yup(z)=g\circ L(k).
\end{eqnarray*}
The fourth identity follows from composition properties of linear maps. Hence $g$ preserves each homogeneous space of $V$.

Let $V^{(r)}$ denote the eigenspace of $g$ with eigenvalue $e^{2\pi i\frac{r}{T}}$, we have $U=\oplus_{r\neq0}V^{(r)}$. Now suppose $u\in V^{(r)}$, we have
\begin{eqnarray*}
c(u)=c\circ g(u)=e^{2\pi i\frac{r}{T}}c(u),\rho(u)=\rho\circ g(u)=e^{2\pi i\frac{r}{T}}\rho(u).
\end{eqnarray*}
Thus, we have $c(u)=\rho(u)=0.$ By linearity, we have $c(U)=\rho(U)=0.$

Now suppose there is $u\in V^{(r)}$ such that $\Delta_k(u)\notin V\otimes U+U\otimes V,$ then $\Delta_k(u)$ has summands of elements of $V^{(0)}\otimes V^{(0)}$. Now we have
\begin{eqnarray*}
e^{2\pi i\frac{r}{T}}\Delta_k(u)=\Delta_k(g(u))=g\otimes g\circ\Delta_k(u)
\end{eqnarray*}
Writing $\Delta_k(u)$ into sum of linear independent elements from $V^{(t)}\otimes V^{(s)}$ and comparing both sides, we get a contradiction. Hence $\Delta_k(u)\in V\otimes U+U\otimes V.$ By linearity, we have $\Delta_k(U)\subseteq V\otimes U+U\otimes V.$

From above discussion, we have $U$ is a coideal. Hence, $V^g=V/U$ is a graded vertex operator coalgebra.                                     $\hfill\Box$

\begin{Remark}\label{Rmk2.1}
(i) Note that $V^{(0)}=\{v\in V|g(v)=v\}$ is not a sub graded vertex operator coalgebra of $V$ in general, since $\Yup(z)$ may be not closed on $V^{(0)}$.

(ii) Suppose $v\in V^{(r)},$ from $(g\otimes g)\Yup(z)v=\Yup(z)gv$, we have $$\Delta_k(v)\subseteq\oplus_{s=0}^{T-1}V^{T-s}\otimes V^{s+r},$$ for all $k\in\mathbb{Z}$.
\end{Remark}

Now we introduce the notion of $g$-twisted $V$-comodules. First, for a linear map $\phi:V\rightarrow V$, let $z^\phi:V\rightarrow V\{z\}$ be a linear map which is defined as $$z^\phi(v)=z^{\lambda}(v),$$ where $v$ is an eigenvector of $\phi$ with eigenvalue $\lambda$.

\begin{Definition}
Let $(V,\Yup(z),c,\rho)$ be a graded vertex operator coalgebra and $g\in \Aut V$ such that $o(g)=T$. Set $V^{(r)}=\{v\in V|g(v)=e^{2\pi i\frac{r}{T}}\}$. If $\mathcal {M}$ is a vector space equipped
with a linear map
\begin{eqnarray*}
&\Yup_\mathcal {M}(z):\mathcal {M}\to V\otimes\mathcal {M}[[z, z^{-1}]],\\
m&\mapsto \Yup_\mathcal {M}(z)(m)=\sum_{k\in\mathbb{Z}}(z^{\frac{-\ln g}{2\pi i}}\otimes\Id)\circ\Delta_{\mathcal {M},k}(m)z^{-k-1},
\end{eqnarray*}
where $\ln g(v)=2\pi i\frac{r}{T}v$ for $v\in V^{(r)}$ and extends to $V$ linearly, satisfying the following conditions:

(i) For any $\ m\in \mathcal {M}$, $\Delta_{\mathcal {M},k}(m)\in V\otimes \mathcal {M}$ is a finite sum,
\begin{eqnarray*}
&&\Delta_{\mathcal {M},k}(m),k<<0,\\
&&(c\otimes\Id)\circ\Yup_{\mathcal {M}}(z)=\Id.
\end{eqnarray*}

(ii) For any $m\in\mathcal {M}$, the following identity
    \begin{eqnarray}\label{JIM3.4}
      &&z_0^{-1}\delta(\frac{z_1-z_2}{z_0})(\Id\otimes \Yup_\mathcal {M} (z_2))\circ \Yup_\mathcal {M} (z_1)(m)\nonumber\\
      &&~~~~~~~~~~-z_0^{-1}\delta(\frac{z_2-z_1}{-z_0})(\tau\otimes\Id)\circ(\Id\otimes \Yup_\mathcal {M} (z_1))\circ \Yup_\mathcal {M} (z_2)(m)\nonumber\\
      &&=z_2^{-1}\delta(\frac{z_1-z_0}{z_2})((\frac{z_1-z_0}{z_2})^{\frac{-\ln g}{2\pi i}}\otimes\Id\otimes\Id)\circ(\Yup (z_0)\otimes \Id)\circ \Yup_\mathcal {M} (z_2)(m)
    \end{eqnarray}
  holds on $\mathcal {M}$. This identity is also called twisted Jacobi identity.

  (iii) Write $(\rho\otimes \Id)\circ\Yup_\mathcal {M} (z)=\sum_{k\in\mathbb{Z}}L(k)z^{k-2}$, then $L(k)\in\End(\mathcal {M})$ and the following Virasoro identity
    \begin{eqnarray}\label{VIM3.5}
     [L(k),L(j)]=(k-j)L(k+j)+\frac{k^3-k}{12}\delta_{k+j,0}d
    \end{eqnarray}
  holds on $\mathcal {M}$.

  (iv) The following identities
    \begin{eqnarray}\label{L(1D)M3.6}
     (L(1)\otimes\Id)\circ \Yup_\mathcal {M} (z)=\frac{d}{dz}\Yup_\mathcal {M} (z)=\Yup_\mathcal {M}(z)\circ L(1)-(\Id\otimes L(1))\circ\Yup_\mathcal {M}(z)
    \end{eqnarray}
  holds on $\mathcal {M}$.
Then, $(\mathcal {M},\Yup_\mathcal {M}(z))$ is called a weak $g$-twisted $V$-comodule.

If $g=1$, this reduces to the definition of weak $V$-comodule.

A weak $g$-twisted $V$-comodule $\mathcal {M}$ is called an admissible $g$-twisted $V$-comodule if $\mathcal {M}$ has a $\frac{1}{T}\mathbb{N}$-gradation $\mathcal {M}=\bigoplus_{t\in\frac{1}{T}\mathbb{N}}M_t$ such that
\begin{eqnarray*}\label{CD1}
\Delta_{\mathcal {M},k}M_t\subset \oplus_{r=0}^{T-1}(V^{(r)}\otimes \mathcal {M})_{t+k+1+\frac{r}{T}}.
\end{eqnarray*}

An admissible $g$-twisted $V$-comodule $\mathcal {M}$ is said to be irreducible if $\mathcal {M}$ has no non-trivial sub admissible $g$-twisted $V$-comodule. When an admissible $g$-twisted $V$-comodule $\mathcal {M}$ is a direct sum of irreducible admissible $g$-twisted $V$-comodules, $\mathcal {M}$ is called completely reducible.

$V$ is called $g$-corational if all admissible $g$-twisted $V$-comodules are completely reducible.
\end{Definition}

\begin{Proposition}\label{WC-WA}
Let $(V,\Yup(z),c,\rho)$ be a graded vertex operator coalgebra and $g\in \Aut V$ such that $o(g)=T$. Let $(\mathcal {M},\Yup_\mathcal {M}(z))$ be an admissible $g$-twisted $V$-comodule.

(i) (Weak cocommutativity) For $m\in \mathcal {M},u^*,v^*\in V',m^*\in\mathcal {M}'$, there exist $q\in \mathbb{N}$ such that
\begin{eqnarray*}
&&(z_1-z_2)^q(u^*\otimes v^*\otimes m^*,(\Id\otimes \Yup_\mathcal {M} (z_2))\circ \Yup_\mathcal {M} (z_1)m)\\
&=&(z_1-z_2)^q(u^*\otimes v^*\otimes m^*,(\tau\otimes\Id)\circ(\Id\otimes \Yup_\mathcal {M} (z_1))\circ \Yup_\mathcal {M} (z_2)m).
\end{eqnarray*}

(ii) (Weak coassociativity) For $m\in \mathcal {M},u^*,v^*\in V',m^*\in\mathcal {M}'$, there exist $q\in \mathbb{N}$ such that
\begin{eqnarray*}
&&(z_0+z_2)^q(u^*\otimes v^*\otimes m^*,((z_2+z_0)^{\frac{\ln g}{2\pi i}}\otimes \Id\otimes\Id)\circ(\Yup (z_0)\otimes \Id)\circ \Yup_\mathcal {M} (z_2)m)\\
&=&(z_0+z_2)^q(u^*\otimes v^*\otimes m^*,((z_2+z_0)^{\frac{\ln g}{2\pi i}}\otimes \Id\otimes\Id)(\Id\otimes \Yup_\mathcal {M} (z_2))\circ \Yup_\mathcal {M} (z_0+z_2)m).
\end{eqnarray*}
\end{Proposition}
\proof The proof are similar to the proofs of Proposition 2.6 and 2.9 in \cite{H1}.

(i) Suppose $u^*,v^*$ are homogeneous, then we have $(u^*\otimes v^*,\Yup(z_0)v)\in\mathbb{C}z_0^{-K}[[z_0]]$ for some $K\in\mathbb{N}$ and all $v\in V.$ Now in twisted Jacobi identity \ref{JIM3.4}, taking $\Res_{z_0}z_0^{K}$, we get the desired equation.

(ii) Similarly, we have $(u^*\otimes m^*,(z_1^{\frac{\ln g}{2\pi i}}\otimes \Id)\circ\Yup_\mathcal {M}(z_1)m)\in\mathbb{C}z_1^{-K}[[z_1]]$ for some $K\in\mathbb{N}$ and all $m\in \mathcal {M}.$ Now in twisted Jacobi identity \ref{JIM3.4}, taking $\Res_{z_1}z_1^{K}z_1^{\frac{\ln g}{2\pi i}}\otimes\Id\otimes\Id$, we get the desired equation.                                                           $\hfill\Box$

\begin{Proposition}\label{Prop3.7}
Let $(V,\Yup (z),c,\rho)$ be a graded vertex operator coalgebra and $(\mathcal {V},Y,\textbf{1},\omega)$ a vertex operator algebra. Let $(\mathcal {M},\Yup_\mathcal {M}(z))$ be an admissible $g$-twisted $V$-comodule and $(\mathcal {N},Y_\mathcal {N})$ an admissible $h$-twisted $\mathcal {V}$-module, where $g\in\Aut V,h\in\Aut\mathcal {V}$ with $o(g)=T_g,o(h)=T_h$. Suppose $V_0\nsubseteq\ker c,V_2\nsubseteq\ker\rho.$ Then we have

(i) $\mathcal {N}'=\oplus_{s\in\mathbb{N}}N_s^*$ is an admissible $h^{-1}$-twisted $\mathcal {V}'$-comodule with $$(\Yup_{\mathcal {N}'}(z)n^*,v\otimes n)=(n^*,Y_\mathcal {N}(v,z)n),$$where $n^*\in \mathcal {N}',v\in\mathcal {V},n\in \mathcal {N}.$

(ii) $\mathcal {M}'=\oplus_{s\in\mathbb{N}}M_s^*$ is an admissible $g^{-1}$-twisted $V'$-module with $$(Y_{\mathcal {M}'}(v^*,z)m^*,m)=(v^*\otimes m^*,\Yup_\mathcal {M}(z)m),$$where $v^*\in V',m^*\in \mathcal {M}',m\in \mathcal {M}.$
\end{Proposition}
\proof (i) Let $n^*\in\mathcal {N}',u,v\in \mathcal {V},n\in\mathcal {N}$, by definitions, we have
\begin{eqnarray*}
((\Id\otimes \Yup_\mathcal {\mathcal {N}'} (z_2))\circ \Yup_{\mathcal {N}'} (z_1)n^*,u\otimes v\otimes n)&=&(n^*,Y_\mathcal {N}(u,z_1)Y_\mathcal {N}(v,z_2)n),\\
((\tau\otimes\Id)\circ(\Id\otimes \Yup_\mathcal {\mathcal {N}'} (z_1))\circ \Yup_{\mathcal {N}'} (z_2)n^*,u\otimes v\otimes n)&=&(n^*,Y_\mathcal {N}(v,z_2)Y_\mathcal {N}(u,z_1)n),\\
((\Yup (z_0)\otimes \Id)\circ \Yup_{\mathcal {N}'} (z_2)n^*,u\otimes v\otimes n)&=&(n^*,Y_\mathcal {N}(Y(u,z_0)v,z_2)n).
\end{eqnarray*}
Since $\mathcal {N}$ is an admissible $h$-twisted $\mathcal {V}$-module, suppose $u\in\mathcal {V}^{(r)}$, we have
\begin{eqnarray*}
&&z_0^{-1}\delta(\frac{z_1-z_2}{z_0})Y_\mathcal {N}(u,z_1)Y_\mathcal {N}(v,z_2)-z_0^{-1}\delta(\frac{-z_2+z_1}{z_0})Y_\mathcal {N}(v,z_2)Y_\mathcal {N}(u,z_1)\nonumber\\
&&\ \ \ \ \ =z_2^{-1}\delta(\frac{z_1-z_0}{z_2})(\frac{z_1-z_0}{z_2})^{\frac{-r}{T_h}}Y_\mathcal {N}(Y(u,z_0)v,z_2),\\
&&\ \ \ \ \ =z_2^{-1}\delta(\frac{z_1-z_0}{z_2})Y_\mathcal {N}(\cdot,z_2)\circ (Y(\cdot,z_0)\otimes\Id)\circ((\frac{z_1-z_0}{z_2})^{\frac{-\ln h}{2\pi i}}\otimes\Id\otimes\Id)u\otimes v\otimes n.
\end{eqnarray*}
Hence, we get
    \begin{eqnarray*}
      &&z_0^{-1}\delta(\frac{z_1-z_2}{z_0})(\Id\otimes \Yup_{\mathcal {N}'} (z_2))\circ \Yup_{\mathcal {N}'} (z_1)n^*\\
      &&~~~~~~~~~~-z_0^{-1}\delta(\frac{z_2-z_1}{-z_0})(\tau\otimes\Id)\circ(\Id\otimes \Yup_{\mathcal {N}'} (z_1))\circ \Yup_{\mathcal {N}'} (z_2)n^*\\
      &&=z_2^{-1}\delta(\frac{z_1-z_0}{z_2})((\frac{z_1-z_0}{z_2})^{\frac{-\ln h^{-1}}{2\pi i}}\otimes\Id\otimes\Id)\circ(\Yup (z_0)\otimes \Id)\circ \Yup_{\mathcal {N}'} (z_2)n^*.
    \end{eqnarray*}
This implies the twisted Jacobi identity holds for $h^{-1}$.

Now let $n^*\in N_t,u\in\mathcal {V}^{(r)},n\in\mathcal {N}$. Since $\mathcal {N}$ is an admissible $h$-twisted $\mathcal {V}$-module, if
\begin{eqnarray*}
&&(\Delta_{\mathcal {N}',k}(n^*),u\otimes n)=(\Res_zz^kz^{\frac{\ln h^{-1}}{2\pi i}}\otimes\Id\circ\Yup(z)n^*,u\otimes n)\\
&=&\Res_zz^{k+\frac{r}{T_h}}(n^*,Y_\mathcal {N}(u,z)n)=(n^*,u_{k+\frac{r}{T_h}}n)\neq0,
\end{eqnarray*}
we have $t=\wt u+\wt n-k-1-\frac{r}{T_h}$. This means $$\Delta_{\mathcal {N}',k}(n^*)\in\oplus_{r=0}^{T-1}(V^{(r)}\otimes \mathcal {M})_{t+k+1+\frac{r}{T_h}},$$ i.e., the $\frac{1}{T}\mathbb{N}$-gradation condition holds.

Other axioms are trivial. Hence $\mathcal {N}'$ is an admissible $h^{-1}$-twisted $\mathcal {V}'$-comodule.

(ii) The proof of (ii) is similar, we omit it.                                                              $\hfill\Box$

\begin{Proposition}\label{Prop3.9}
Let $(V,\Yup (z),c,\rho)$ be a graded vertex operator coalgebra, and $(\mathcal {M},\Yup_\mathcal {M}(z))$ an admissible $g$-twisted $V$-comodule with $o(g)=T.$ Then, $(\mathcal {M},\Yup_\mathcal {M}(z))$ is an admissible $V^g$-comodule.
\end{Proposition}
\proof Define $\Yup_{V^g,\mathcal {M}(z)}:\mathcal {M}\rightarrow V^g\otimes \mathcal {M}$ by $(\pi\otimes\Id)\circ\Yup_\mathcal {M}(z)$, where $\pi$ is the canonical projection, we have $\mathcal {M}$ is an admissible $g$-twisted $V^g$-comodule. For any $v\in V^g$, we have $g(v)=v$. This implies $T=r=0$ in the definition of admissible $g$-twisted $V$-comodules. Hence, $(\mathcal {M},\Yup_{V^g,\mathcal {M}}(z))$ is an admissible $V^g$-comodule.                               $\hfill\Box$

\section{Coassociative coalgebra $C_g(V)$}
In this section, we introduce the coassociative coalgebras $C_g(V)$, where $V$ is a graded vertex operator coalgebra and $g\in\Aut V$ is a finite order automorphism of order $T$.

\begin{Definition}
Let $V$ be a graded vertex operator coalgebra and $g\in\Aut V$ a finite order automorphism of order $T$. Define
\begin{equation}\label{def}
 C_g(V)=\{v\in V|\Res_z(\frac{(1+z)^{L(0)-1+\delta_0+\frac{\ln g}{2\pi i}}}{z^{1+\delta_0}}\otimes \Id)\circ\Yup(z)v=0\},
\end{equation}
where $\delta_i:V\rightarrow V$ is a linear map defined by $\delta_i(u)=u$ if $u\in V^{(i)}$, and $\delta(u)=0$ if $u\in V^{(r)}$ with $r\neq i.$
\end{Definition}

\begin{Remark}\label{Rmk4.2}
(i) It is obvious that $C_g(V)$ is nonempty and a vector space.

(ii) If $g=\Id$, then $T=0$. In this case, we have $C_g(V)$ is the untwisted coassociative coalgebra $C(V)$ defined in \cite{W}.

(iii) From the proof of Proposition \ref{Prop3.4}, it is easy to see that $c\circ\delta_0=c,~\rho\circ\delta_0=\rho,$ and $c\circ\delta_i=0,~\rho\circ\delta_i=0,$ if $i\neq0.$

(iv) Recall the canonical projection $\pi:V\rightarrow V^g,$ it is obvious that $\pi\circ\delta_0=\pi$, and $\pi\circ\delta_i=0$ if $i\neq0.$
\end{Remark}

\begin{Lemma}\label{Lm4.3}
Suppose $v\in V^{(r)}$ is nonzero with $r\neq0$, then $v\notin C_g(V)$.
\end{Lemma}
\proof From the proof of Proposition \ref{Prop3.4}, we know $c(V^{(r)})=0$ for $r\neq0$. Hence, for $v\in V^{(r)}$, if $v\in C_g(V)$, we have
  \begin{eqnarray*}
  0&=&\Res_z(\Id\otimes c)\circ(\frac{(1+z)^{L(0)-1+\delta_0+\frac{\ln g}{2\pi i}}}{z^{1+\delta_0}}\otimes \Id)\circ\Yup(z)v\\
  &=&\Res_z(\frac{(1+z)^{L(0)-1+\delta_0+\frac{\ln g}{2\pi i}}}{z^{1+\delta_0}}\otimes \Id)\circ(\Id\otimes c)\circ\Yup(z)v\\
  &=&\Res_z\sum_{j\geq0}(\tbinom{L(0)-1+\delta_0+\frac{\ln g}{2\pi i}}{j}z^{j-1-\delta_0}\otimes\Id)\circ(\Id\otimes c)\circ \Yup(z)v\\
  &=&\Res_z\sum_{j\geq0}(\tbinom{L(0)-1+\frac{r}{T}}{j}z^{j-1}\otimes\Id)\circ(\Id\otimes c)\circ \Yup(z)v\\
  &=&(\Id\otimes c)\circ \Delta_{-1}(v)=v.
  \end{eqnarray*}
The first identity follows from definition. The second identity follows from composition properties of linear maps. The fourth identity follows from the properties of linear map $c$ and Remark \ref{Rmk2.1}. This is impossible. Hence $v\notin C_g(V).$                                                                                                $\hfill\Box$

From Lemma \ref{Lm4.3}, the restriction of $\pi:V\rightarrow V^g$ on $C_g(V)$ is injective, still denoted by
\begin{eqnarray}\label{pi-C}
\pi:C_g(V)\rightarrow V^g,
\end{eqnarray}
see Definition \ref{Def3.3}. Hence we can view $C_g(V)$ as a subspace of $V^g.$

\begin{Lemma}\label{Lm4.4}
Suppose $v\in V^{(0)}$. If $\Res_z(\frac{(1+z)^{L(0)-1+\delta_0+\frac{\ln g}{2\pi i}}}{z^{1+\delta_0}}\otimes \Id)\circ\Yup(z)v=0,$ then $L(1)v+L(0)v=0.$
\end{Lemma}
\proof First, we know $\Res_z(\Id\otimes c)\circ(\frac{(1+z)^{L(0)-1+\delta_0+\frac{\ln g}{2\pi i}}}{z^{1+\delta_0}}\otimes \Id)\circ\Yup(z)v=0.$ On the other hand,
  \begin{eqnarray*}
  &&\Res_z(\Id\otimes c)\circ(\frac{(1+z)^{L(0)-1+\delta_0+\frac{\ln g}{2\pi i}}}{z^{1+\delta_0}}\otimes \Id)\circ\Yup(z)v\\
  &=&\Res_z(\frac{(1+z)^{L(0)-1+\delta_0+\frac{\ln g}{2\pi i}}}{z^{1+\delta_0}}\otimes \Id)\circ(\Id\otimes c)\circ\Yup(z)v\\
  &=&\Res_z\sum_{j\geq0}(\tbinom{L(0)-1+\delta_0+\frac{\ln g}{2\pi i}}{j}z^{j-1-\delta_0}\otimes\Id)\circ(\Id\otimes c)\circ \Yup(z)v\\
  &=&\Res_z\sum_{j\geq0}(\tbinom{L(0)}{j}z^{j-2}\otimes\Id)\circ(\Id\otimes c)\circ \Yup(z)v\\
  &=&(\Id\otimes c)\circ\Delta_{-2}(v)+(L(0)\otimes c)\circ\Delta_{-1}(v)\\
  &=&(\Id\otimes c)\circ(L(1)\otimes\Id)\circ\Delta_{-1}(v)+(L(0)\otimes c)\circ\Delta_{-1}(v)\\
  &=&(L(1)+L(0))\circ(\Id\otimes c)\circ\Delta_{-1}(v)\\
  &=&L(1)v+L(0)v.
  \end{eqnarray*}
The first and sixth identities follows from composition properties of linear maps. The third and fourth identities follows from the properties of linear map $c$ and Remark \ref{Rmk2.1}. The fifth identity follows from $L(1)$-derivation.
Hence, $L(1)v+L(0)v=0.$                           $\hfill\Box$

\begin{Lemma}\label{4.6}
Let $v\in V^{(0)}$. Then $v\in C_g(V)$, if and only if
\begin{equation*}
  \left\{ \begin{aligned}
           &\Res_z(\delta_0\otimes\delta_0)\circ(\frac{(1+z)^{L(0)}}{z^2}\otimes\Id)\circ\Yup(z)v=0, \\
           &\Res_z(\delta_r\otimes\delta_{-r})\circ(\frac{(1+z)^{L(0)-1+\frac{r}{T}}}{z}\otimes\Id)\circ\Yup(z)v=0,
                            \end{aligned} \right.
                            \end{equation*}
for all $r=1,2,...,T-1.$
\end{Lemma}
\proof Since $v\in V^{(0)}$, we have $\Yup(z)v\in\oplus_{r=0}^{T-1}V^{(r)}\otimes V^{(-r)}[[z,z^{-1}]].$ Then the "only if" part follows from this direct sum.

Conversely, the "if" part follows from $\sum_{i=0}^{T-1}\delta_i\otimes\delta_i=\Id$ on $\oplus_{r=0}^{T-1}V^{(v)}\otimes V^{(-r)}$.                        $\hfill\Box$

\begin{Definition}
For $v\in C_g(V)$, define
\begin{equation}\label{cop}
\Delta_g(v)=\Res_z(\frac{\delta_0\circ(1+z)^{L(0)}}{z}\otimes\Id)\circ\Yup(z)v.
\end{equation}
\end{Definition}

Recall the coassociative coalgebra $(C(V^g),\Delta,c)$ associated to graded vertex operator coalgebra $V^g$, see \cite{W}, which is defined as follows:
$$C(V^g)=\{v\in V^g | \Res_z(\frac{(1+z)^{L(0)}}{z^2}\otimes\Id)\circ\Yup_{V^g}(z)v=0\},$$ for any $v\in C(V^g)$,
$$\Delta(v)=\Res_z(\frac{(1+z)^{L(0)}}{z}\otimes\Id)\circ\Yup_{V^g}(z)v,~c(v)=c_{V^g}(v).$$

\begin{Proposition}\label{Prop4.1}
$\pi$ induces a linear map from $C_g(V)$ to $C(V^g)$, still denoted by $\pi: C_g(V)\rightarrow C(V^g).$ Furthermore, for any $v\in C_g(V)$, we have $$(\pi\otimes\pi)\circ\Delta_g(v)=\Delta\circ\pi(v),~c(v)=c_{V^g}\circ\pi(v).$$ In other words, if $(C_g(V).\Delta_g,c)$ is a coassociative coalgebra, then, $\pi$ induces a homomorphism of coassociative coalgebras, and $C_g(V)$ is a sub coassociative coalgebra of $C(V^g)$.
\end{Proposition}
\proof Let $v\in C_g(V)$, i.e., $\Res_z(\frac{(1+z)^{L(0)-1+\delta_0+\frac{\ln g}{2\pi i}}}{z^{1+\delta_0}}\otimes \Id)\circ\Yup(z)v=0.$ Now we have
  \begin{align*}
    ~~&\Res_z(\frac{(1+z)^{L(0)}}{z^2}\otimes\Id)\circ\Yup_{V^g}(z)\pi(v) \\
    =~~&(\pi\otimes\pi)\circ\Res_z(\frac{(1+z)^{L(0)}}{z^2}\otimes\Id)\circ\Yup(v)\\
    =~~&(\pi\otimes\pi)\circ\Res_z(\frac{(1+z)^{L(0)-1+\delta_0+\frac{\ln g}{2\pi i}}}{z^{1+\delta_0}}\otimes \Id)\circ\Yup(z)v\\
    =~~&0.
  \end{align*}
The first identity follows since $\pi$ is homomorphism of graded vertex operator coalgebras, and $\pi$ commute with $L(0)$.
The second identity follows from the properties of projection $\pi$ which states $V^{(r)}\subseteq\ker\pi$ for all $r\neq0$. Hence $\im\pi\subseteq C(V^g)$.

Now, it is trivial $c(v)=c_{V^g}\circ\pi(v)$, and
  \begin{align*}
    ~~&\Delta\circ\pi(v)\\
    =~~&\Res_z(\frac{(1+z)^{L(0)}}{z}\otimes\Id)\circ\Yup_{V^g}(z)\pi(v) \\
    =~~&(\pi\otimes\pi)\circ\Res_z(\frac{\delta_0\circ(1+z)^{L(0)}}{z}\otimes\Id)\circ\Yup(v)\\
    =~~&(\pi\otimes\pi)\circ\Delta_g(v).
  \end{align*}
Hence $\pi$ is a homomorphism under the assumption $C_g(V)$ is a coassociative coalgebra.                         $\hfill\Box$

\begin{Lemma}\label{Lm4.7}
For $v\in C_g(V)$, we have $$\Res_z(\frac{(1+z)^{L(0)-1+\delta_0+\frac{\ln g}{2\pi i}+k}}{z^{1+\delta_0+l}}\otimes \Id)\circ\Yup(z)v=0,$$for $l\geq k\geq0.$
\end{Lemma}
\proof Since $$\frac{(1+z)^{L(0)-1+\delta_0+\frac{\ln g}{2\pi i}+k}}{z^{1+\delta_0+l}}=\sum_{i\geq0}\tbinom{k}{i}\frac{(1+z)^{L(0)-1+\delta_0+\frac{\ln g}{2\pi i}}}{z^{1+\delta_0+l-i}}.$$ By linearity, it is enough to show the identity for $l\geq0.$

Using induction, it is true for $l=0$ by definition. Suppose $$\Res_z(\frac{(1+z)^{L(0)-1+\delta_0+\frac{\ln g}{2\pi i}}}{z^{1+\delta_0+l}}\otimes \Id)\circ\Yup(z)v=0,$$ now we have
    \begin{eqnarray*}
      0&=&(L(1)\otimes \Id)\circ\Res_z(\frac{(1+z)^{L(0)-1+\delta_0+\frac{\ln g}{2\pi i}}}{z^{1+\delta_0+l}}\otimes \Id)\circ\Yup (z)v\\
      &=&\Res_z(\frac{(1+z)^{L(0)+\delta_0+\frac{\ln g}{2\pi i}}}{z^{1+\delta_0+l}}\otimes \Id)\circ (L(1)\otimes \Id)\circ\Yup (z)v\\
      &=&\Res_z(\frac{(1+z)^{L(0)+\delta_0+\frac{\ln g}{2\pi i}}}{z^{1+\delta_0+l}}\otimes \Id)\circ\frac{d}{dz}\Yup (z)v\\
      &=&-\Res_z(\frac{d}{dz}\frac{(1+z)^{L(0)+\delta_0+\frac{\ln g}{2\pi i}}}{z^{1+\delta_0+l}}\otimes \Id)\circ\Yup (z)v\\
      &=&-\Res_z((L(0)+\delta_0+\frac{\ln g}{2\pi i})\otimes \Id)\circ(\frac{(1+z)^{L(0)+\delta_0+\frac{\ln g}{2\pi i}-1}}{z^{1+\delta_0+l}}\otimes \Id)\circ\Yup (z)v\\
      &&+\Res_z(1+\delta_0+l)(\frac{(1+z)^{L(0)+\delta_0+\frac{\ln g}{2\pi i}}}{z^{1+\delta_0+l+1}}\otimes \Id)\circ\Yup (z)v\\
      &=&-\Res_z((L(0)+\delta_0+\frac{\ln g}{2\pi i})\otimes \Id)\circ(\frac{(1+z)^{L(0)-1+\delta_0+\frac{\ln g}{2\pi i}}}{z^{1+\delta_0+l}}\otimes \Id)\circ\Yup (z)v\\
      &&+\Res_z(1+\delta_0+l)(\frac{(1+z)^{L(0)-1+\delta_0+\frac{\ln g}{2\pi i}}}{z^{1+\delta_0+l}}\otimes \Id)\circ\Yup (z)v\\
      &&+\Res_z(1+\delta_0+l)(\frac{(1+z)^{L(0)-1+\delta_0+\frac{\ln g}{2\pi i}}}{z^{1+\delta_0+l+1}}\otimes \Id)\circ\Yup (z)v.
    \end{eqnarray*}
The second identity follows since $L(1)\circ g=g\circ L(1)$ by Proposition \ref{Prop3.4}, and $L(1)$ is homogeneous of degree $-1$. In last identity, by assumption, the first and second terms are equal to $0$, this implies the last term is also $0.$ This completes the induction.                                             $\hfill\Box$

\begin{Lemma}\cite{W}\label{L(0)delta}
Let $V$ be a graded vertex operator coalgebra. Then, we have $$\Delta_k\circ L(0)=(L(0)\otimes L(0)-k-1)\circ\Delta_k.$$
\end{Lemma}

\begin{Lemma}\label{Lm4.9}
For $v\in C_g(V)$, we have $\Delta_g(v)\in C_g(V)\otimes C_g(V)$, and $$(\Id\otimes \Delta_g)\circ \Delta_g(v)=(\Delta_g\otimes \Id)\circ \Delta_g(v).$$
\end{Lemma}
\proof Since $C_g(V)$ is a subspace of $V$, we have $C_g(V)\otimes C_g(V)=V\otimes C_g(V)\bigcap C_g(V)\otimes V$.

First, we show $\Delta_g(v)\in V\otimes C_g(V)$ for any $v\in C_g(V)$. To prove this relation, using identity \ref{def}, it is enough to show
$$(\Id\otimes \Res_{z_2}\frac{(1+z_2)^{L(0)-1+\delta_0+\frac{\ln g}{2\pi i}}}{z_2^{1+\delta_0}}\otimes\Id)\circ (\Id\otimes\Yup (z_2))\circ\Delta_g(v)=0.$$Now, we have
  \begin{eqnarray*}
    &&(\Id\otimes \Res_{z_2}\frac{(1+z_2)^{L(0)-1+\delta_0 +\frac{\ln g}{2\pi i}}}{z_2^{1+\delta_0 }}\otimes\Id)\circ(\Id\otimes\Yup (z_2))\circ \Delta_g(v)\\
    &=&(\Res_{z_1}\frac{\delta_0\circ(1+z_1)^{L(0)}}{z_1}\otimes \Res_{z_2}\frac{(1+z_2)^{L(0)-1+\delta_0 +\frac{\ln g}{2\pi i}}}{z_2^{1+\delta_0 }}\otimes\Id)\circ(\Id\otimes\Yup (z_2))\circ \Yup (z_1)v\\
    &=&\Res_{z_0,z_1,z_2}(\frac{\delta_0\circ(1+z_1)^{L(0)}}{z_1}\otimes \frac{(1+z_2)^{L(0)-1+\delta_0 +\frac{\ln g}{2\pi i}}}{z_2^{1+\delta_0 }}\otimes\Id)\circ\\
    &&\{-z_0^{-1}\delta_0(\frac{z_2-z_1}{-z_0})(\tau\otimes\Id)\circ(\Id\otimes \Yup (z_1))\circ \Yup (z_2)v+z_1^{-1}\delta_0(\frac{z_2+z_0}{z_1})(\Yup (z_0)\otimes \Id)\circ \Yup (z_2)v\}\\
    &=& (\tau\otimes\Id_V)\circ\Res_{z_1,z_2}(\frac{(1+z_2)^{L(0)-1+\delta_0 +\frac{\ln g}{2\pi i}}}{z_2^{1+\delta _0 }}\otimes\frac{\delta_0\circ(1+z_1)^{L(0)}}{z_1}\otimes\Id)\circ(\Id\otimes \Yup (z_1))\circ \Yup (z_2)v\\
    && +\Res_{z_0,z_2}(\frac{\delta_0\circ(1+z_2+z_0)^{L(0)}}{z_2+z_0}\otimes\frac{(1+z_2)^{L(0)-1+\delta_0 +\frac{\ln g}{2\pi i}}}{z_2^{1+\delta_0 }}\otimes\Id)\circ (\Yup (z_0)\otimes \Id)\circ \Yup (z_2)v\\
    &=& (\tau\otimes\Id)\circ\Res_{z_1,z_2}(\Id\otimes\frac{\delta_0\circ(1+z_1)^{L(0)}}{z_1}\otimes\Id)\circ(\frac{(1+z_2)^{L(0)-1+\delta_0 +\frac{\ln g}{2\pi i}}}{z_2^{1+\delta_0 }}\otimes \Yup (z_1))\circ \Yup (z_2)v\\
    && +\Res_{z_0,z_2}\sum_{i\geq0,j\geq0,l}\tbinom{L(0)}{i}\tbinom{-1}{j}(\delta_0\circ(1+z_2)^{L(0)-i}\otimes\frac{(1+z_2)^{L(0)-1+\delta_0 +\frac{\ln g}{2\pi i}}}{z_2^{1+\delta_0 +1+j}}\otimes\Id)\circ \\
    &&(\Delta_l\otimes \Id)\circ \Yup (z_2)vz_0^{-l-1+i+j}
    \end{eqnarray*}
  \begin{eqnarray*}
    &=& (\tau\otimes\Id)\circ\Res_{z_1,z_2}(\Id\otimes\frac{\delta_0\circ(1+z_1)^{L(0)}}{z_1}\otimes\Id)\circ(\Id\otimes \Yup (z_1))\circ\\
    &&(\frac{(1+z_2)^{L(0)-1+\delta_0 +\frac{\ln g}{2\pi i}}}{z_2^{1+\delta_0 }}\otimes \Id)\circ \Yup (z_2)v\\
    && +\Res_{z_2}\sum_{i\geq0,j\geq0}\tbinom{L(0)}{i}\tbinom{-1}{j}(\delta_0\circ(1+z_2)^{L(0)-i}\otimes\frac{(1+z_2)^{L(0)-1+\delta_0 +\frac{\ln g}{2\pi i}}}{z_2^{1+\delta_0 +1+j}}\otimes\Id)\circ\\
    && (\Delta_{i+j}\otimes \Id)\circ \Yup (z_2)v\\
    &=&\delta_0\otimes\Id\otimes\Id\circ(\Res_{z_2}\sum_{i\geq0,j\geq0}\tbinom{L(0)}{i}\tbinom{-1}{j}(\Delta_{i+j}\otimes \Id))\circ(\frac{(1+z_2)^{L(0)+\delta_0 +\frac{\ln g}{2\pi i}+j}}{z_2^{1+\delta_0 +1+j}}\otimes\Id)\circ \Yup (z_2)v\\
    &=&0.
  \end{eqnarray*}
The first identity follows from definition of $\Delta_g$ and composition properties of linear maps. The second identity follows from Jacobi identity \ref{JI3.4}. Third identity follows from properties of $\tau$ and definition of $\Res$. The fifth identity follows from composition properties of linear maps and definition of $\Res$. The first term of the sixth identity disappears from the definition of $C_g(V)$, The second term follows from Lemma \ref{L(0)delta}. Last identity follows from Lemma \ref{Lm4.7}.

Now by definition of $\Delta_g$ and above proof, we can assume $\Delta_g(v)\in V^{(0)}\otimes C_g(V)$.

Second, we show $\Delta_g(v)\in C_g(V)\otimes V$ for any $v\in C_g(V)$. To prove this relation, using identity \ref{def}, it is enough to show $$(\Res_{z_1}\frac{(1+z_1)^{L(0)-1+\delta_0+\frac{\ln g}{2\pi i}}}{z_1^{1+\delta_0}}\otimes\Id \otimes\Id)\circ (\Yup (z_1)\otimes\Id)\circ\Delta_g(v)=0.$$ From Lemma \ref{4.6}, this is equivalent to prove
 \begin{equation*}
  \left\{ \begin{aligned}
           &\Res_z(\delta_0\otimes\delta_0\otimes\Id)\circ(\frac{(1+z)^{L(0)}}{z^2}\otimes\Id\otimes\Id)\circ(\Yup(z)\otimes\Id)\circ\Delta_{g}(v)=0, \\
           &\Res_z(\delta_r\otimes\delta_{-r}\otimes\Id)\circ(\frac{(1+z)^{L(0)-1+\frac{r}{T}}}{z}\otimes\Id\otimes\Id)\circ(\Yup(z)\otimes\Id)\circ\Delta_{g}(v)=0,
                            \end{aligned} \right.
                            \end{equation*}
for all $r=1,...,T-1.$

 The proof of first identity is similar to the untwisted case, we prove the second identity. Let $r\neq0$, for $v\in V^{(0)}\cap C_g(V)$, applying $\delta_r\otimes\delta_{-r}\otimes\Id$ to Lemma \ref{Lm4.7}, we have$$\Res_z(\delta_r\otimes\delta_{-r}\otimes\Id)\circ(\frac{(1+z)^{L(0)-1+\frac{r}{T}+k}}{z^{1+l}}\otimes \Id)\circ\Yup(z)v=0,$$for $l\geq k\geq0.$ Now we have
  \begin{align*}
    &\Res_z(\delta_r\otimes\delta_{-r}\otimes\Id)\circ(\frac{(1+z)^{L(0)-1+\frac{r}{T}}}{z}\otimes\Id\otimes\Id)\circ(\Yup(z)\otimes\Id)\circ\Delta_{g}(v)\\
    =&\Res_{z_2}(\delta_r\otimes\delta_{-r}\otimes\Id)\circ(\sum_{i\geq0}\tbinom{L(0)-1+\frac{r}{T}}{i}\otimes\Id\otimes\Id)\circ(\Delta_{i-1}\otimes\Id)\circ (\frac{(1+z_2)^{L(0)-1}}{z_2}\otimes\Id)\circ\Yup(z_2)v\\
    =&\Res_{z_2}(\delta_r\otimes\delta_{-r}\otimes\Id)\circ(\sum_{i\geq0}\tbinom{L(0)-1+\frac{r}{T}}{i}\otimes\Id\otimes\Id)\circ\\
    &~~~~((1+z_2)^{L(0)-i-1+\frac{r}{T}}\otimes\frac{(1+z_2)^{L(0)-\frac{r}{T}}}{z_2}\otimes\Id)\circ(\Delta_{i-1}\otimes\Id)\circ\Yup(z_2)v
        \end{align*}
    \begin{align*}
    =&\Res_{z_0,z_1,z_2}(\delta_r\otimes\delta_{-r}\otimes\Id)\circ(\sum_{i\geq0}\tbinom{L(0)-1+\frac{r}{T}}{i}(1+z_2)^{L(0)-i-1+\frac{r}{T}}z_0^{i-1}\otimes\frac{(1+z_2)^{L(0)-\frac{r}{T}}}{z_2}\otimes\Id)\circ\\
    &~~~~\{z_0^{-1}\delta(\frac{z_1-z_2}{z_0})(\Id\otimes \Yup (z_2))\circ \Yup (z_1)v-z_0^{-1}\delta(\frac{z_2-z_1}{-z_0})(\tau\otimes\Id)\circ(\Id\otimes \Yup (z_1))\circ \Yup (z_2)v\}  \\
    =&\Res_{z_1,z_2}(\delta_r\otimes\delta_{-r}\otimes\Id)\circ(\frac{(1+z_1)^{L(0)-1+\frac{r}{T}}}{z_1-z_2}\otimes\frac{(1+z_2)^{L(0)-\frac{r}{T}}}{z_2}\otimes\Id)\circ(\Id\otimes \Yup (z_2))\circ \Yup (z_1)v\\
    &-\Res_{z_1,z_2}(\tau\otimes\Id)\circ(\delta_{-r}\otimes\delta_{r}\otimes\Id)\circ(\frac{(1+z_2)^{L(0)-\frac{r}{T}}}{z_2}\otimes\frac{(1+z_1)^{L(0)-1+\frac{r}{T}}}{-z_2+z_1}\otimes\Id)\circ \\
    &~~~~(\Id\otimes \Yup (z_1))\circ \Yup (z_2)v  \\
    =&\Res_{z_1,z_2}(\delta_r\otimes\delta_{-r}\otimes\Id)\circ(\sum_{i\geq0}\tbinom{-1}{i}(-1)^i\Id\otimes\frac{(1+z_2)^{L(0)-\frac{r}{T}}}{z_2^{1-i}}\otimes\Id)\circ(\Id\otimes \Yup (z_2))\circ \\
    &~~~~~~(\frac{(1+z_1)^{L(0)-1+\frac{r}{T}}}{z_1^{1+i}}\otimes\Id)\circ\Yup (z_1)v\\
    &-(\tau\otimes\Id)\circ\Res_{z_1,z_2}(\delta_{-r}\otimes\delta_{r}\otimes\Id)\circ(\sum_{i\geq0}\tbinom{-1}{i}(-1)^{-1-i}\Id\otimes\frac{(1+z_1)^{L(0)-1+\frac{r}{T}}}{z_1^{-i}}\otimes\Id)\circ \\
    &~~~~(\Id\otimes \Yup (z_1))\circ (\frac{(1+z_2)^{L(0)-1-\frac{r}{T}+1}}{z_2^{1+i+1}}\otimes\Id)\circ\Yup (z_2)v  \\
    &=0.
  \end{align*}
Hence, $\Delta_g(v)\in C_g(V)\otimes V$ for any $v\in C_g(V)$.

From above two steps, we can see that $\Delta_g(v)\in C_g(V)\otimes C_g(V)$ for any $v\in C_g(V)$.

The proof of the coassociativity is similar.  Let $v\in C_g(V)$, we have
  \begin{eqnarray*}
    &&(\Delta_g\otimes \Id)\circ \Delta_g(v)=(\Res_{z_1}\frac{\delta_0\circ(1+z_1)^{L(0)}}{z_1}\otimes\Id\otimes\Id)\circ (\Yup (z_1)\otimes\Id)\circ\Delta_g(v)\\
    &=&\sum_{i\geq0}(\delta_0\circ\tbinom{L(0)}{i}\otimes\Id \otimes\Id)\circ (\Delta_{i-1}\otimes\Id)\circ\Res_{z_2}(\frac{\delta_0\circ(1+z_2)^{L(0)}}{z_2}\otimes \Id)\circ\Yup (z_2)v\\
    &=&(\delta_0\otimes\delta_0\otimes\Id)\circ\Res_{z_0,z_1,z_2}\sum_{i\geq0}(\tbinom{L(0)}{i}\otimes\Id \otimes\Id)\circ \\
     &&(\frac{(1+z_2)^{L(0)-i}}{z_2}\otimes (1+z_2)^{L(0)}\otimes \Id)\circ z_1^{-1}\delta(\frac{z_2+z_0}{z_1})(\Yup (z_0)\otimes\Id)\circ\Yup (z_2)vz_0^{i-1}\\
    &=&(\delta_0\otimes\delta_0\otimes\Id)\circ\{\Res_{z_1,z_2}(\frac{(1+z_1)^{L(0)}}{z_1-z_2}\otimes \frac{(1+z_2)^{L(0)}}{z_2}\otimes \Id)\circ(\Id\otimes \Yup (z_2))\circ \Yup (z_1)v\\
    &&+\Res_{z_1,z_2}(\frac{(1+z_1)^{L(0)}}{-z_2+z_1}\otimes \frac{(1+z_2)^{L(0)}}{z_2}\otimes \Id)\circ(\tau\otimes\Id)\circ(\Id\otimes \Yup (z_1))\circ \Yup (z_2)v\}\\
    &=&(\delta_0\otimes\delta_0\otimes\Id)\circ\{\Res_{z_1,z_2}\sum_{i\geq0}(\tbinom{-1}{i}(-1)^i\frac{(1+z_1)^{L(0)}}{z_1^{1+i}}\otimes \frac{(1+z_2)^{L(0)}}{z_2^{1-i}}\otimes \Id)\circ \\
    &&~~~~(\Id\otimes \Yup (z_2))\circ \Yup (z_1)v
                 \end{eqnarray*}
    \begin{eqnarray*}
    &&+\Res_{z_1,z_2}(\tau\otimes\Id_V)\circ\sum_{i\geq0}(\tbinom{-1}{i}(-1)^{-1-i}\frac{(1+z_2)^{L(0)}}{z_2^{2+i}}\otimes \frac{(1+z_1)^{L(0)}}{z_1^{-i}}\otimes \Id)\circ\\
    &&~~~~(\Id\otimes \Yup (z_1))\circ \Yup (z_2)v\}\\
    &=&(\delta_0\otimes\delta_0\otimes\Id)\circ\{\Res_{z_1,z_2}\sum_{i\geq0}(\tbinom{-1}{i}(-1)^i\Id\otimes \frac{(1+z_2)^{L(0)}}{z_2^{1-i}}\otimes \Id)\circ(\Id\otimes \Yup (z_2))\circ\\
    &&~~~~(\frac{(1+z_1)^{L(0)}}{z_1^{1+i}}\otimes\Id)\circ \Yup (z_1)v\\
    &&+\Res_{z_1,z_2}(\tau\otimes\Id)\circ\sum_{i\geq0}(\tbinom{-1}{i}(-1)^{-1-i}\Id\otimes \frac{(1+z_1)^{L(0)}}{z_1^{-i}}\otimes \Id)\circ\\
    &&~~~~(\Id\otimes \Yup (z_1))\circ(\frac{(1+z_2)^{L(0)}}{z_2^{2+i}}\otimes \Id)\circ \Yup (z_2)v\}\\
    &=&\Res_{z_1,z_2}(\Id\otimes \frac{\delta_0\circ(1+z_2)^{L(0)}}{z_2}\otimes \Id)\circ(\Id\otimes \Yup (z_2))\circ(\frac{\delta_0\circ(1+z_1)^{L(0)}}{z_1}\otimes \Id)\circ \Yup (z_1)v\\
    &=&(\Id\otimes \Res_{z_2}\frac{\delta_0\circ(1+z_2)^{L(0)}}{z_2}\otimes\Id)\circ (\Id\otimes\Yup (z_2))\circ\Delta_g(v)\\
    &=&(\Id\otimes \Delta_g)\circ \Delta_g(v).
  \end{eqnarray*}
The third identity follows from properties of linear map $\delta_0$ and composition properties of linear maps. The seventh identity follows from Lemma \ref{Lm4.7} and definitions of $\delta_0$, since in the sixth identity, all terms belong to $V^{(0)}\otimes V^{(0)}\otimes V\{z_1,z_2\}$. Others are trivial.

This completes the proof.                             $\hfill\Box$

\begin{Lemma}\label{Lm4.6}
For $v\in C_g(V)$, we have $$(\Id\otimes c)\circ\Delta_g(v)=(c\otimes\Id)\circ\Delta_g(v)=v.$$
\end{Lemma}
\proof First, from identity \ref{cop}, for homogeneous $v$ of weight $t$, we have
\begin{eqnarray*}
&&(c\otimes\Id)\circ\Delta_g(v)\\
&=&\Res_z(c\otimes \Id)\circ(\frac{\delta_0\circ(1+z)^{L(0)}}{z}\otimes\Id)\circ\Yup(z)v\\
&=&(c\circ\delta_0\otimes \Id)\circ\Res_z\frac{(1+z)^{L(0)}}{z}\otimes \Id\circ\Yup (z)v\\
&=&(c\otimes \Id)\circ\Res_z\frac{(1+z)^{L(0)}}{z}\otimes \Id(\sum_{k\in\mathbb{Z}}\sum\sum_{s\in\mathbb{Z}}v'_{t+k+1-s}\otimes v''_{s}z^{-k-1})\\
&=&(c\otimes \Id)\circ\Res_z\sum_{k\in\mathbb{Z}}\sum\sum_{s\in\mathbb{Z}}\frac{(1+z)^{t+k+1-s}}{z}v'_{t+k+1-s}\otimes v''_{s}z^{-k-1}\\
&=&\Res_z\sum_{k\in\mathbb{Z}}\sum\sum_{s\in\mathbb{Z}}\frac{(1+z)^{t+k+1-s}}{z}c(v'_{t+k+1-s})\otimes v''_{s}z^{-k-1}\\
&=&\Res_z\sum\frac{1}{z}1\otimes v=v.
\end{eqnarray*}
The third identity follows since $c\circ\delta_0=c,$ and $\Delta_k$ is homogeneous of degree $k+1.$ The fifth identity follows from properties of $c$.

Second, from identity \ref{cop}, using the injection $\pi:C_g(V)\rightarrow V^g$, we have
\begin{eqnarray*}
(\Id\otimes c)\circ\Delta_g(v)&=&(\Id\otimes c)\circ(\pi^{-1}\otimes \pi^{-1})\circ(\pi\otimes\pi)\circ\Delta_g(v)\\
&=&\pi^{-1}\circ\Res_z(\Id\otimes c)\circ(\frac{(1+z)^{L(0)}}{z}\otimes\Id)\circ\Yup_{V^g}(z)v\\
&=&\pi^{-1}\circ\Res_z(\frac{(1+z)^{L(0)}}{z}\otimes\Id)\circ(\Id\otimes c)\circ\Yup_{V^g}(z)v\\
&=&\pi^{-1}\circ \pi(v)=v.
\end{eqnarray*}
The first identity follows from $\Delta_g(v)\in C_g(V)\otimes C_g(V)$. The second identity follows since $c\circ\pi=c$. The fourth identity follows from properties of $c$ and $V^g$ is a graded vertex operator algebra. Here $\pi^{-1}$ should be considered as $\pi^{-1}|_{\im\pi}.$

This completes the proof.                                                                  $\hfill\Box$

\begin{Theorem}\label{Thm4.3}
Suppose $\dim C_g(V)<\infty$, then $(C_g(V),\Delta_g,c)$ is a coassociative coalgebra, where $c=c|_{C_g(V)}$, which is still denoted by $c$.
\end{Theorem}
\proof From Lemma \ref{Lm4.9}, we know $\Delta_g$ is well-defined and satisfies coassociativity. From Lemma \ref{Lm4.6}, we know $c$ satisfies the counit identities. Thus $(C_g(V),\Delta_g,c)$ is a coassociative coalgebra.                                                      $\hfill\Box$

\begin{Remark}
From Proposition \ref{Prop4.1}, we know $C_g(V)$ is a subspace of coassociative coalgebra $C(V^g)$. Hence to prove $C_g(V)$ is a coassociative coalgebra, it is enough to prove $\Delta_g(v)\in C_g(V)\otimes C_g(V)$ for all $v\in C_g(V).$ This will simplify the proof of Theorem \ref{Thm4.3} a lot.
\end{Remark}

\section{$C_g(V)$-comodules}
Let $(V,\Yup(z),c,\rho)$ be a graded vertex operator coalgebra and $g\in\Aut V$ a finite order automorphism of order $T$. Let $(\mathcal {M},\Yup_\mathcal {M}(z))$ be an admissible $g$-twisted $V$-comodule with $M_0\neq0$. Suppose $\dim C_g(V)<\infty$.

\begin{Lemma}\label{Lm5.1}
Let $m\in M_0$ and $i\in\mathbb{Q}$. If $i<-1$, we have $$\Res_z(z^{L(0)+i}\otimes \Id)\circ \Yup_\mathcal {M}(z)m=0.$$
\end{Lemma}
\proof By assumption, we have
  \begin{eqnarray*}
    &&\Res_z(z^{L(0)+i}\otimes \Id)\circ \Yup_\mathcal {M}(z)m\\
    &=&\Res_z(z^{L(0)+i}\otimes \Id)\circ\sum_{k\in\mathbb{Z}}(z^{\frac{-\ln g}{2\pi i}}\otimes\Id)\circ\Delta_{\mathcal {M},k}(m)z^{-k-1}\\
    &=&\Res_z\sum_{k\in\mathbb{Z}}\sum\sum_{s\in\mathbb{\mathbb{Z}}}v'^{,r}_{s}\otimes m''^{,r}_{k+1+\frac{r}{T}-s}z^{-k-1-\frac{r}{T}+s+i}\\
    &=&0,
  \end{eqnarray*}
where $\Delta_{\mathcal {M},k}(v)=\sum v'^{,r}_{s}\otimes m''^{,r}_{k+1+\frac{r}{T}-s}$ with $v_s'^{,r}\in V_s^{(r)}$. Last identity follows since $k+1+\frac{r}{T}-s\geq0$.                      $\hfill\Box$

\begin{Theorem}\label{CM0}
For any $m\in M_0$, define $$\Delta_{M_0}(m)=\Res_z(z^{L(0)-1}\otimes \Id)\circ \Yup_\mathcal {M}(z)m.$$Then $\Delta_{M_0}(m)\in C_g(V)\otimes M_0$, and $(M_0,\Delta_{M_0})$ gives a $C_g(V)$-comodule structure on $M_0.$
\end{Theorem}
\proof Since $(\mathcal {M},\Yup_\mathcal {M}(z))$ is an admissible $g$-twisted $V$-comodule, it is also an admissible $V^g$-comodule by Proposition \ref{Prop3.9}. From \cite{W}, we know $M_0$ is a $C(V^g)$-comodule with $$\Delta_{M_0}(m)=\Res_z(z^{L(0)-1}\otimes \Id)\circ \Yup_{V^g,\mathcal {M}}(z)m.$$
From Proposition \ref{Prop4.1}, we know $C_g(V)$ is a sub coassociative coalgebra of $C(V^g)$. Hence if $\Delta_{M_0}(m)\in C_g(V)\otimes M_0$ for all $m\in M_0$, $M_0$ is automatically a $C_g(V)$-comodule.

First, we prove $\Delta_{M_0}(m)\in V\otimes M_0$ for any $m\in M_0.$ This follows from
  \begin{eqnarray*}
    \Delta_{M_0}(m)&=&\Res_z(z^{L(0)-1}\otimes \Id)\circ \Yup_\mathcal {M}(z)m\\
    &=&\Res_z(z^{L(0)-1}\otimes \Id)\circ\sum_{k\in\mathbb{Z}}(z^{\frac{-\ln g}{2\pi i}}\otimes\Id)\circ\Delta_{\mathcal {M},k}(m)z^{-k-1}\\
    &=&\Res_z\sum_{k\in\mathbb{Z}}\sum\sum_{s\in\mathbb{\mathbb{Z}}}v'^{,r}_{s}\otimes m''^{,r}_{k+1+\frac{r}{T}-s}z^{-k-1-\frac{r}{T}+s-1}\\
    &=&\sum_{k\in\mathbb{Z}}\sum v'_{k+1}\otimes m''_{0}\in V\otimes M_0.
  \end{eqnarray*}
From above computation, we also have $r=0$. This means $\Delta_{M_0}(m)\in V^{(0)}\otimes M_0.$

Second, we prove $\Delta_{M_0}(m)\in C_g(V)\otimes M$ for any $m\in M_0.$ From Lemma \ref{4.6}, this is equivalent to prove
 \begin{equation*}
  \left\{ \begin{aligned}
           &\Res_z(\delta_0\otimes\delta_0\otimes\Id)\circ(\frac{(1+z)^{L(0)}}{z^2}\otimes\Id\otimes\Id)\circ(\Yup(z)\otimes\Id)\circ\Delta_{M_0}(m)=0, \\
           &\Res_z(\delta_r\otimes\delta_{-r}\otimes\Id)\circ(\frac{(1+z)^{L(0)-1+\frac{r}{T}}}{z}\otimes\Id\otimes\Id)\circ(\Yup(z)\otimes\Id)\circ\Delta_{M_0}(m)=0,
                            \end{aligned} \right.
                            \end{equation*}
for all $r=1,...,T-1.$

For the first identity, since $\pi|_{V^{(0)}}:V^{(0)}\rightarrow V^g$ is bijective and $M_0$ is a $C(V^g)$-comodule, we have $$\Res_z(\pi\otimes\pi\otimes\Id)\circ(\delta_0\otimes\delta_0\otimes\Id)\circ(\frac{(1+z)^{L(0)}}{z^2}\otimes\Id\otimes\Id)\circ(\Yup(z)\otimes\Id)\circ\Delta_{M_0}(m)=0.$$This proves the first identity.

For the second identity, applying $\delta_r\otimes\delta_{-r}\otimes\Id$ to twisted Jacobi identity \ref{JIM3.4}, the results of both sides belong to $V^{(r)}\otimes V^{(-r)}\otimes \mathcal {M}\{z_0,z_1,z_2\}$. Let $r\neq0$, using Lemma \ref{Lm5.1}, applying $$\Res_{z_0.z_1}(\delta_r\otimes\delta_{-r}\otimes\Id)\circ z_0^{-1}(z_1^{L(0)-1+\frac{r}{T}}\otimes\Id\otimes\Id)$$ to twisted Jacobi identity \ref{JIM3.4}. Since $m\in M_0,$ we have
 \begin{align*}
 -&\Res_{z_0.z_1}(\delta_r\otimes\delta_{-r}\otimes\Id)\circ z_0^{-2}\delta(\frac{z_2-z_1}{-z_0})(\tau\otimes\Id)\circ(\Id\otimes z_1^{L(0)-1+\frac{r}{T}}\otimes\Id)\circ(\Id\otimes \Yup_\mathcal {M} (z_1))\circ \Yup_\mathcal {M} (z_2)m\\
 =&\Res_{z_0.z_1}(\delta_r\otimes\delta_{-r}\otimes\Id)\circ z_2^{-1}\delta(\frac{z_1-z_0}{z_2})((z_2+z_0)^{\frac{r}{T}}z_1^{L(0)-1}\otimes\Id\otimes\Id)\circ(\Yup (z_0)\otimes \Id)\circ \Yup_\mathcal {M} (z_2)m.
 \end{align*}
Applying $\Res_{z_2}\Id\otimes z_2^{L(0)-\frac{r}{T}}\otimes\Id$ to above identity, we have
 \begin{align*}
 &0\\
 =&\Res_{z_0.z_2}(\delta_r\otimes\delta_{-r}\otimes\Id)\circ ((z_2+z_0)^{L(0)-1+\frac{r}{T}}\otimes z_2^{L(0)-\frac{r}{T}}\otimes\Id)\circ(\Yup (z_0)\otimes \Id)\circ \Yup_\mathcal {M} (z_2)m\\
 =&\Res_{z_2}(\delta_r\otimes\delta_{-r}\otimes\Id)\circ(\sum_{i\geq0}\tbinom{L(0)-1+\frac{r}{T}}{i}z_2^{L(0)-1-i}\otimes z_2^{L(0)}\otimes\Id)\circ(\Delta_{i-1}\otimes \Id)\circ \Yup_\mathcal {M} (z_2)m\\
 =&\Res_{z_2}(\delta_r\otimes\delta_{-r}\otimes\Id)\circ(\sum_{i\geq0}\tbinom{L(0)-1+\frac{r}{T}}{i}\Delta_{i-1}\otimes\Id)\circ(z_2^{L(0)-1}\otimes \Id)\circ \Yup_\mathcal {M} (z_2)m\\
 =&\Res_z(\delta_r\otimes\delta_{-r}\otimes\Id)\circ(\frac{(1+z)^{L(0)-1+\frac{r}{T}}}{z}\otimes\Id\otimes\Id)\circ(\Yup(z)\otimes\Id)\circ\Delta_{M_0}(m).
 \end{align*}
Hence, $\Delta_{M_0}(m)\in C_g(V)\otimes M,$ and this completes the proof.                                                            $\hfill\Box$

\begin{Proposition}\label{P4.3}
Let $(V,\Yup (z),c,\rho)$ be a graded vertex operator coalgebra and $g\in\Aut V$ of order $T$, $(\mathcal {M}^1,\Yup_{\mathcal {M}^1}(z))$ and $(\mathcal {M}^2,\Yup_{\mathcal {M}^2}(z))$ two $g$-twisted admissible $V$-comodules. If $\psi:\mathcal {M}^1\rightarrow \mathcal {M}^2$ is a $g$-twisted $V$-comodule homomorphism, then $\psi|_{M^1_0}:M^1_0\rightarrow M^2_0$ is a $C_g(V)$-comodule homomorphism. Thus we get a functor $\Omega_g$ from admissible admissible $g$-twisted $V$-comodules category to $C_g(V)$-comodules category defined by $\Omega_g(\mathcal {M})=M_0$.
\end{Proposition}
\proof From Proposition \ref{Prop3.9}, $\mathcal {M}^1,\mathcal {M}^2$ are also two admissible $V^g$-comodules. From \cite{W}, we know $\psi|_{M^1_0}$ is a $C(V^g)$-comodule homomorphism. Hence it is also a $C_g(V)$-comodule homomorphism.
                                                   $\hfill\Box$

\section{Admissible $g$-twisted $V$-comodules}
Let $(V,\Yup(z),c,\rho)$ be a graded vertex operator coalgebra and $g\in\Aut V$ a finite order automorphism of order $T$. From section 4, we have a coassociative coalgebra $C_g(V)$ if $\dim C_g(V)<\infty$. This implies $C_g(V)^*$ is an associative algebra. On the other hand, from Proposition \ref{VC-VA} and Proposition \ref{Prop3.2}, we know $(V',Y(\cdot,z),c|_{V_0},\rho|_{V_0})$ is a vertex operator algebra and $g^{-1}\in \Aut V'$ is also a finite order automorphism of $V'$ of order $T$. Recall from \cite{DLM2} the twisted associative algebra $A_{g^{-1}}(V')$ which is defined as follows: Let $V'^{,(r)}=\{V^*\in V'| g^{-1}(v^*)=e^{\frac{2\pi ir}{T}}\}$. set
\begin{equation*}
O_{g^{-1}}(V')=\{u^*\circ_{g^{-1}}v^*=\Res_z\frac{(1+z)^{\wt u^*-1+\delta_0+\frac{r}{T}}}{z^{1+\delta_0}}Y(u^*,z)v^*| u^*\in V'^{,(r)},v^*\in V'\},
\end{equation*}
where $\delta_i:V'\rightarrow V'$ is induced from $\delta_i:V\rightarrow V,$ i.e., $\delta_i(u^*)=u^*$ if $u^*\in V'^{,(i)}$, and $\delta_i(u^*)=0$ if $u^*\notin V'^{,(i)}$. Set $A_{g^{-1}}(V')=V'/O_{g^{-1}}(V')$, for $u^*\in V'^{,(r)}$, set
\begin{equation*}
u^**_{g^{-1}}v^*=  \left\{ \begin{aligned}
           &\Res_z\frac{(1+z)^{\wt u^*}}{z}Y(u^*,z)v^*,~&r=0, \\
           &0,~&r\neq0, \\
                            \end{aligned} \right.
\end{equation*}
If $\dim A_{g^{-1}}(V')<\infty$, we have another coassociative coalgebra $A_{g^{-1}}(V')^*.$

\begin{Lemma}\label{Lm6.1}
$v\in C_g(V)$ if and only if $(v^*,v)=0$ for all $v^*\in O_{g^{-1}}(V').$
\end{Lemma}
\proof Let $v\in C_g(V),u^*\circ_{g^{-1}} v^*\in O_{g^{-1}}(V')$. By linearity, we can assume $u^*\in V'^{,(r)}$ is homogeneous. Now we have
  \begin{eqnarray*}
   (u^*\circ_{g^{-1}} v^*,v)&=&(\Res_z\frac{(1+z)^{\wt u^*-1+\delta_0+\frac{r}{T}}}{z^{1+\delta_0}}Y(u^*,z)v^*,v)\\
   &=&(\Res_zY(\cdot,z)\circ(\frac{(1+z)^{L(0)-1+\delta_0+\frac{\ln g^{-1}}{2\pi i}}}{z^{1+\delta_0}}\otimes\Id)(u^*\otimes v^*),v)\\
   &=&(u^*\otimes v^*,\Res_z(\frac{(1+z)^{L(0)-1+\delta_0+\frac{\ln g}{2\pi i}}}{z^{1+\delta_0}}\otimes\Id)\circ\Yup(z)v)=0.
  \end{eqnarray*}
The third identity follows since $(V'^{,(r)},V^{(s)})=0$ for $r\neq s.$

Conversely, if $v\notin C_g(V)$, we have $\Res_z(\frac{(1+z)^{L(0)-1+\delta_0+\frac{\ln g}{2\pi i}}}{z^{1+\delta_0}}\otimes\Id)\circ\Yup(z)v\in V\otimes V[[z,z^{-1}]]$ is nonzero. Hence there exist $u^*,v^*\in V'$ such that $$(u^*\otimes v^*,\Res_z(\frac{(1+z)^{L(0)-1+\delta_0+\frac{\ln g}{2\pi i}}}{z^{1+\delta_0}}\otimes\Id)\circ\Yup(z)v)\neq0.$$ Now we have an element $f\circ g\in O_{g^{-1}}(V')$ such that $(u^*\circ_{g^{-1}}v^*,v)\neq0.$                                         $\hfill\Box$

\begin{Proposition}\label{Prop5.2A-C}
(i) Suppose $\dim C_g(V)<\infty$, there is a well-defined surjective homomorphism $\Phi_g:A_{g^{-1}}(V')\rightarrow C_g(V)^*$ of algebras.

(ii) Suppose $\dim A_{g^{-1}}(V')<\infty$, there is a well-defined epimorphism $\Psi_g:C_g(V)\rightarrow A_{g^{-1}}(V')^*$ of coalgebras.

(iii) If both $C_{g^{-1}}(V)$ and $A_{g^{-1}}(V')$ are finite dimensional, then $\Phi_g$ and $\Psi_g$ are isomorphic.
\end{Proposition}
\proof (i) Since $C_g(V)$ is a sub space of $V$, we have a linear map $V'\rightarrow V^*\rightarrow C_g(V)^*,$ denote it by $\phi_g.$ By Lemma \ref{Lm5.1}, we know $O_{g^{-1}}(V')\subseteq\ker\phi.$ Hence, we get a well-defined linear map $\Phi_g:A_{g^{-1}}(V')\rightarrow C_g(V)^*.$

Now $\forall u^*,v^*\in A_{g^{-1}}(V'), v\in V$, we have
  \begin{eqnarray*}
   &&(\Phi_g(u^**_{g^{-1}}v^*),v)=(u^**_{g^{-1}}v^*,\Delta_g(v))=\sum(u^*,v')(v^*,v'')\\
   &=&\sum(\Phi_g(u^*),v')(\Phi_g(v^*),v'')=(\Phi_g(u^*)*\Phi_g(v^*),v).
  \end{eqnarray*}
Hence, $\Phi_g$ is a homomorphism.

If $\dim C_g(V)<\infty$, we know there is $k\in\mathbb{N}$ such that $C_g(V)\subseteq\oplus_{i\leq k}V_i$. This implies the linear map $\phi_g:V'\rightarrow C_g(V)^*$ is surjective. Hence $\Phi_g$ is epimorphic.

(ii) Since $A_{g^{-1}}(V')^*=\{v\in (V')^*|(v^*,v)=0,\forall v^*\in O(V')\}.$ By Lemma \ref{Lm5.1} and $V\rightarrow (V')^*$, there is a well-defined linear map $\Psi_g:C_g(V)\rightarrow A_{g^{-1}}(V')^*.$ Similar to the proof of (i), we can prove the rest statements of (ii).

(iii) Suppose $\dim A_{g^{-1}}(V')<\infty$, then we have a monomorphism $\Psi_g^*:A_{g^{-1}}(V')^{**}=A_{g^{-1}}(V')\rightarrow C_g(V)^*$ by (ii). Furthermore, $\Psi_g^*=\Phi_g$. Hence, they are isomorphisms.                                                                                   $\hfill\Box$

\begin{Lemma}\label{Lm5.3}
Let $V$ be a graded vertex operator coalgebra and $g\in\Aut V$ a finite order automorphism of order $T$. If $\dim A_{g^{-1}}(V')<\infty$, then we have $\dim C_g(V)<\infty.$
\end{Lemma}
\proof Suppose $\dim A_{g^{-1}}(V')<\infty,$ then there is $k\in\mathbb{N}$, such that $V_i^*\subseteq O_{g^{-1}}(V')$ for all $i>k.$ This means that for any $v^*\in V_i^*,v\in C_g(V)$ with $i>k$, we have $(v^*,v)=0.$ Hence $C_g(V)\subseteq\oplus_{i\leq k}V_i,$ which implies $\dim C_g(V)<\infty.$                            $\hfill\Box$

\begin{Theorem}
Let $V$ be a graded vertex operator coalgebra and $g\in\Aut V$ a finite order automorphism of order $T$. If $\dim A_{g^{-1}}(V')<\infty$, then $\Phi_g$ and $\Psi_g$ are isomorphisms.
\end{Theorem}
\proof This follows from Proposition \ref{Prop5.2A-C} and Lemma \ref{Lm5.3}.                    $\hfill\Box$

Now, let $(M,\Delta_M)$ be a left $C_g(V)$-comodule. Then, $M^*$ is a left $C_g(V)^*$-module. By Proposition \ref{Prop5.2A-C}, it is also an $A_{g^{-1}}(V')$-module. By Dong-Li-Mason's theory, we get an admissible $g^{-1}$-twisted $V'$-module $L_{g^{-1}}(M^*)$ satisfying $L_{g^{-1}}(M^*)_0=M^*.$ We give the construction of $L_{g^{-1}}(M^*)$ briefly.

Consider $\widehat{V'}[g^{-1}]=\oplus_{r=0}^{T-1}V'^{,(r)}\otimes t^{\frac{r}{T}}\C[t,t^{-1}]/D(\oplus_{r=0}^{T-1}V'^{,(r)}\otimes t^{\frac{r}{T}}\C[t,t^{-1}])$, where $D=L(-1)\otimes 1+1\otimes\frac{d}{dt}$. Then $\widehat{V'}[g^{-1}]$ is a $\frac{1}{T}\Z$-graded Lie algebra with triangular decomposition $$\widehat{V'}[g^{-1}]=\widehat{V'}[g^{-1}](-)\oplus\widehat{V'}[g^{-1}](0)\oplus\widehat{V'}[g^{-1}](+).$$ Furthermore, there is a well-defined Lie algebra epimorphism $\phi:\widehat{V'}[g^{-1}](0)\rightarrow A_{g^{-1}}(V')$. So $M^*$ is a $\widehat{V'}[g^{-1}](0)$-module. Regard $M^*$ as trivial $\widehat{V'}[g^{-1}](-)$-module. Set $$U_{g^{-1}}(M^*)=\mathcal {U}(\widehat{V'}[g^{-1}])\otimes_{\mathcal {U}(\widehat{V'}[g^{-1}](0)+\widehat{V'}[g^{-1}](-))}M^*\cong\mathcal {U}(\widehat{V'}[g^{-1}](+))\otimes M^*,$$ and $L_{g^{-1}}(M^*)=U_{g^{-1}}(M^*)/J$, where $$J=\{m^*\in U_{g^{-1}}(M^*)|<m^{**},fm^*>=0,\forall m^{**}\in M^{**},f\in\mathcal {U}(\widehat{V'}[g^{-1}])\}.$$
Then $(L_{g^{-1}}(M^*),Y_{M^*})$ is an admissible $g^{-1}$-twisted $V'$-module with $$Y_{M^*}(u^*,z)=\sum_{k\in\frac{r}{T}+\mathbb{Z}}u^*\otimes t^kz^{-k-1},$$ for any $u^*\in V'^{,(r)}$.

From above discussion, we know $L_{g^{-1}}(M^*)'$ is a left admissible $g$-twisted $V''=V$-comodule. On the other hand, there is an injection $M\rightarrow M^{**}=L_{g^{-1}}(M^*)_0^*\rightarrow L_{g^{-1}}(M^*)'$. If $\dim M<\infty,$ the first injection is also a bijection. Under above injection, we can regard $M$ as a subspace of $L_{g^{-1}}(M^*)'.$ Define $\mathcal {W}_g(M)$ to be the sub $g$-twisted $V$-comodule of $L_{g^{-1}}(M^*)'$ generated by $M$, i.e., the smallest sub $g$-twisted $V$-comodule containing $M$. Define $\mathcal {L}_g(M)$ to be the quotient $\mathcal {W}_g(M)/\mathcal {J}$, where $\mathcal {J}$ is the maximal sub $g$-twisted $V$-comodule of $\mathcal {W}_g(M)$ which intersects $M$ trivially.

\begin{Lemma}\cite{Z}
Let $\mathcal {N}$ be an admissible $g^{-1}$-twisted $V'$-module, then $N_0$ is an $A_{g^{-1}}(V')$-module with module structure defined as $\varrho_{N_0}(f\otimes n)=f\otimes t^{\wt f-1}n,$ if $f\in V'^{,(0)}$ is homogeneous and extend to all $f\in V'^{,(0)}$ by linearity., and $\varrho_{N_0}(f\otimes n)=0$ if $f\notin V'^{,(0)}$.
\end{Lemma}

\begin{Lemma}\label{Lm6.1}
Suppose $\dim M$ is countable. Then, we have $\dim\mathcal {L}_g(M)$ is countable. Furthermore, $\mathcal {L}_g(M)_0=M.$
\end{Lemma}
\proof Define $F^0(M)=M$, suppose we have already defined $F^k(M)$, define $F^{k+1}(M)$ to be the space spanned by the second tensor factors of all $$\Delta_i(m),i\in\mathbb{Z},m\in F^k(M).$$
From properties of $c$, we can see that $F^k(M)\subseteq F^{k+1}(M)$ for all $k$. Hence, we get the following filtration $$F^0(M)\subseteq F^1(M)\subseteq\cdots\subseteq F^k(M)\subseteq\cdots.$$
Since $\Delta_i(m)$ is a finite sum, if $\dim M$ is countable, using induction, we can see that $\dim F^k(M)$ is countable for all $k.$

Now set $F=\cup_kF^k(M)$, we have $\dim F$ is countable. Since $L_{g^{-1}}(M^*)'$ is an admissible $g$-twisted $V$-comodule, it is obvious that $F$ is also an admissible $g$-twisted $V$-comodule, such that $M\subseteq F\subseteq\mathcal {W}_g(M).$ Hence, $\mathcal {W}_g(M)=F$ and $\dim\mathcal {W}_g(M)$ is countable. Now $\mathcal {L}_g(M)$ is a quotient comodule of $\mathcal {W}_g(M)$, it is obvious that $\dim\mathcal {L}_g(M)$ is countable.

To show $\dim\mathcal {L}_g(M)_0=M,$ it is enough to show that $\mathcal {W}_g(M)_0=M$, which is equivalently to show for any $m\in\mathcal {W}_g(M)$, when the second tensor factors of $\Delta_i(m)$ belong to $\mathcal {W}_g(M)_0$ for any $i\in\mathbb{Z}$, they belong to $M$. Using above filtration, we prove that $\forall m\in F^k(M)$, when the second tensor factors of $\Delta_i(m)$ belong to $\mathcal {W}_g(M)_0$, they belong to $M$. By linearity, we can assume $m$ is homogeneous.

Induction on $k$. If $k=0$, for any $m\in M$, then $$\Delta_i(m)=\sum m'\otimes m''\in \oplus_r(V^{(r)}\otimes\mathcal {W}_g(M))_{i+1+\frac{r}{T}}.$$ If $m''\in\mathcal {W}_g(M)_0$, then $r=0,~m'\in V_{i+1}$. Hence for any $v^*\in V_{i+1}^{*,(0)},m^*\in M^*$, we have
  \begin{eqnarray*}
   &&(\Delta_i(m),v^*\otimes m^*)=(m,(v^*\otimes t^i)m^*)\\
   &=&(m,\varrho_{M^*}(v^*\otimes m^*))=(\Delta_M(m),v^*\otimes m^*).
  \end{eqnarray*}
The first identity follows from $(V'^{,(r)},V^{(s)})=0$ for $r\neq s.$ Since $(M,\Delta_M)$ is a $C_g(V)$-comodule, if $m'\in V_{i+1}$, then $m''\in M.$

Suppose it is true for $m\in F^k(M)$. Now let $m\in F^{k+1}(M).$ By definition, there are $m_s'\in V,m_s''\in \mathcal {W}_g(M)$ with $m_0''=m$ such that $\sum m_s'\otimes m_s''=\Delta_p(x)$ for some $p\in\mathbb{Z}, x\in F^k(M)$. Suppose $m_s'$ are linear independent, let $f\in V'$ be the dual element of $m_0'$, then we have
  \begin{eqnarray*}
   \Delta_i(m)&=&\Delta_i(\sum f(m_s')m_s'')=(\Id\otimes\Delta_i)\circ(f\otimes \Id)\circ\Delta_p(x)\\
   &=&(f\otimes \Id\otimes\Id)\circ(\Id\otimes \Delta_i)\circ\Delta_p(x).
  \end{eqnarray*}
Since $\mathcal {W}_g(M)$ is an admissible $g$-twisted $V$-comodule, by Proposition \ref{WC-WA}, we have the following weak coassociativity
\begin{eqnarray*}
&&(z_0+z_2)^q(u^*\otimes v^*\otimes m^*,((z_2+z_0)^{\frac{\ln g}{2\pi i}}\otimes \Id\otimes\Id)\circ(\Yup (z_0)\otimes \Id)\circ \Yup_\mathcal {M} (z_2)m)\\
&=&(z_0+z_2)^q(u^*\otimes v^*\otimes m^*,((z_2+z_0)^{\frac{\ln g}{2\pi i}}\otimes \Id\otimes\Id)(\Id\otimes \Yup_\mathcal {M} (z_2))\circ \Yup_\mathcal {M} (z_0+z_2)m).
\end{eqnarray*}
for some $q\in\mathbb{N}$, and all $u^*,v^*\in V', m^*\in \mathcal {W}_g(M)'$. Hence we can see that $(\Id\otimes \Delta_i)\circ\Delta_p(x)$ is a linear combination of $(\Delta_j\otimes \Id)\circ\Delta_l(x)$. By induction, we can see that if the second tensor factors of $\Delta_l(x)$ belong to $\mathcal {W}_g(M)_0$, they belong to $M$. Hence the statement is also true for $\Delta_i(m)$. Thus, we get $\mathcal {W}_g(M)_0=M.$ This means $\mathcal {L}_g(M)_0=M$ by definition.

This completes the proof.                                                            $\hfill\Box$

\begin{Proposition}
$\mathcal {L}_g$ is a functor from $C_g(V)$-comodules category to admissible $g$-twisted $V$-comodules category such that $\Omega_g\circ\mathcal {L}_g=\Id.$ Furthermore, $\mathcal {L}_g$ sends simple objects to simple objects.
\end{Proposition}
\proof Let $\psi:M^1\rightarrow M^2$ be a $C_g(V)$-comodule homomorphism, then $\psi^*:(M^2)^*\rightarrow (M^1)^*$ is a $C_g(V)^*$-module homomorphism. By Proposition \ref{Prop5.2A-C}, $\psi^*$ is also an $A_{g^{-1}}(V')$-module homomorphism. By Dong-Li-Mason's theory, $\psi^*$ induces an admissible $g^{-1}$-twisted $V'$-module homomorphism from $L_{g^{-1}}((M^2)^*)$ to $L_{g^{-1}}((M^1)^*)$, still denoted by $\psi^*.$ Now, $\psi^*$ induces an admissible $g$-twisted $V$-comodule homomorphism $\psi^{**}:L_{g^{-1}}((M^1)^*)'\rightarrow L_{g^{-1}}((M^2)^*)'$ and $\psi^{**}|_{M^1_0}=\psi.$ Define $\mathcal {L}_g(\psi)=\psi^{**}|_{\mathcal {L}_{g^{-1}}(M^1)}$, we get $\mathcal {L}_g(\psi):\mathcal {L}_g(M^1)\rightarrow\mathcal {L}_g(M^2)$ is an admissible $g$-twisted $V$-comodule homomorphism. This proves $\mathcal {L}_g$ is a functor.

$\Omega_g\circ\mathcal {L}_g=\Id$ is obvious by definitions.

Now let $M$ be a simple $C_g(V)$-comodule. If $\mathcal {L}_g(M)$ is not simple, then it has a sub admissible $g$-twisted $V$-comodule $\mathcal {M}=\oplus_{i\geq0}M_i$ such that $M_0\neq0$, Now $M_0$ is a sub $C_g(V)$-comodule of $M$. Hence $M_0=M$. This implies $\mathcal {M}=\mathcal {L}_g(M)$ and $\mathcal {L}_g(M)$ is a simple admissible $g$-twisted $V$-comodule.                                               $\hfill\Box$

\begin{Theorem}
If $V$ is a $g$-corational graded vertex operator coalgebra and $g\in\Aut V$ a finite order automorphism of order $T$ such that $\dim C_g(V)<\infty$, then $C_g(V)$ is a cosemisimple coassociative coalgebra.
\end{Theorem}
\proof Let $M$ be a $C_g(V)$-comodule, then there is an admissible $g$-twisted $V$-comodule $\mathcal {L}_g(M)$ such that $\mathcal {L}_g(M)_0=M$. If $V$ is $g$-corational, then $\mathcal {L}_g(M)$ is cosemisimple. Hence $\mathcal {L}_g(M)_0$ is a cosemisimple $C_g(V)$-comodule, i.e., $M$ is a cosemisimple $C_g(V)$-comodule. This implies $C_g(V)$ is cosemisimple.                                 $\hfill\Box$

\begin{Corollary}
If $V$ is a $g$-corational graded vertex operator coalgebra and $g\in\Aut V$ a finite order automorphism of order $T$ such that $\dim C_g(V)<\infty$, then $V$ has only finitely many inequivalent irreducible admissible $g$-twisted $V$-comodules.
\end{Corollary}

\section{High level coassociative coalgebras}
In this section, we introduce the higher level coassociative coalgebras $C^k_g(V)$ for all $k\in\frac{1}{T}\mathbb{N}$, where $V$ is a graded vertex operator coalgebra and $g\in\Aut V$ is a finite order automorphism of order $T$. For $k\in\frac{1}{T}\mathbb{N}$, there are $l\in\mathbb{N},i=0,1,...,T-1$ such that $k=l+\frac{i}{T}$.

\begin{Definition}
Let $V$ be a graded vertex operator coalgebra and $g\in\Aut V$ a finite order automorphism of order $T$. For $i,r=0,1,...,T-1,$ set
\begin{equation*}
  \left\{ \begin{aligned}
           &\delta_i(r)=1,~i\geq r, \\
           &\delta_i(r)=0,~i<r, \\
           &\delta_i(T)=1. \\
                            \end{aligned} \right.
                            \end{equation*}

Let $k=l+\frac{i}{T}$ with $l\in\mathbb{N},i=0,1,...,T-1$, define
\begin{equation*}
C^k_g(V)=\{v\in V|L(1)v+L(0)v=0,\Res_z(\frac{(1+z)^{L(0)-1+\delta_i(\frac{T\ln g}{2\pi i})+l+\frac{\ln g}{2\pi i}}}{z^{2l+\delta_i(\frac{T\ln g}{2\pi i})+\delta_i(T-\frac{T\ln g}{2\pi i})}}\otimes \Id)\circ\Yup(z)v=0\}.
\end{equation*}

For $v\in C^k_g(V)$, define
\begin{equation*}
\Delta^k_g(v)=\sum_{j=0}^l(-1)^j\tbinom{j+l}{l}\Res_z(\frac{(1+z)^{L(0)+l}}{z^{l+j+1}}\otimes\Id)\circ\Yup(z)v.
\end{equation*}
\end{Definition}

\begin{Remark}
If $g=\Id$, then $T=r=i=0$. In this case, we have $C^k_g(V)$ is the untwisted coassociative coalgebra $C^k(V)$ defined in \cite{W}.
\end{Remark}

\begin{Lemma}\label{Lm7.3}
Suppose $\dim C_g^k(V)<\infty$. Let $v\in C_g^k(V).$ Suppose $s\geq t\geq0$, we have $$\Res_z(\frac{(1+z)^{L(0)-1+l+\delta_i(\frac{T\ln g}{2\pi i})+\frac{\ln g}{2\pi i}+t}}{z^{2l+\delta_i(\frac{T\ln g}{2\pi i})+\delta_i(T-\frac{T\ln g}{2\pi i})+s}}\otimes\Id)\circ\Yup(z)v=0.$$
\end{Lemma}
\proof This is similar to the proof of Lemma \ref{Lm4.7}, we omit it.
  $\hfill\Box$

\begin{Lemma}\cite{W}
For any $v\in C(V^g)$, we have $L(1)v+L(0)v=0.$
\end{Lemma}

\begin{Proposition}
Suppose $\dim C_g^k(V)<\infty$ for all $k\in\frac{1}{T}\mathbb{N}$.

(i) If $k=0$, then $C_g^0(V)=C_g(V)$.

(ii) We have the following filtration of coassociative coalgebras $$C_g^0(V)\subseteq C_g^{\frac{1}{T}}(V)\subseteq\cdots\subseteq C_g^k(V)\subseteq\cdots.$$
\end{Proposition}
\proof (i) It is enough to prove for any $v\in C_g(V)$, $L(1)v+L(0)v=0.$

Suppose $L(1)v+L(0)v\neq0.$ Since $C_g(V)\nsubseteq\ker\pi,$ we have $v\notin\ker\pi=\oplus_{r=1}^{T-1}V^{(r)}$. This implies $L(1)v+L(0)v\notin\oplus_{r=1}^{T-1}V^{(r)}=\ker\pi.$ By Proposition \ref{Prop4.1} and above Lemma, we have$$\pi(L(1)v+L(0)v)=(L(1)+L(0))\pi(v)=0.$$
This is a contradiction. Hence $L(1)v+L(0)v=0.$

(ii) Recall from \cite{W}, there is a filtration of coassociative coalgebras $$C^0(V^g)\subseteq C^1(V^g)\subseteq\cdots.$$Following the proofs of Lemma \ref{Lm4.3}, using Lemma \ref{Lm7.3}, it is easy to see that for $k=l+\frac{i}{T}$ with $i=0,1,...,T-1$, $C_g^k(V)$ are sub coassociative coalgebras of $C^l(V^g)$. Then it is trivial to get the desired filtration.
                 $\hfill\Box$

\begin{Proposition}
Suppose $\dim C_g^k(V)<\infty$.
Let $(\mathcal {M},\Yup_\mathcal {M}(z))$ be an admissible $g$-twisted $V$-comodule. Then $\Omega_g^k(\mathcal {M})=\oplus_{i=0}^kM_i$ is a $C_g^k(V)$-comodule with comodule structure defined as $$\Delta_{\Omega_g^k(\mathcal {M})}m=\Res_z(z^{L(0)-1}\otimes \Id)\circ\Yup_\mathcal {M}(z)m,$$where $m\in\Omega_g^k(\mathcal {M}).$ Furthermore, $\Omega_g^k$ is a functor from admissible $g$-twisted $V$-comodules category to $C_g^k(V)$-comodules category.
\end{Proposition}
\proof  This is similar to section 5, we omit it.                                       $\hfill\Box$

\begin{Proposition}\label{Prop7.4}
Let $V$ be a graded vertex operator coalgebra and $g\in\Aut V$ a finite order automorphism of order $T$. Suppose $\dim C_g^k(V)<\infty$.

(i) There is a well-defined surjective homomorphism $\Phi_g^k:A_{k,g^{-1}}(V')\rightarrow C_g^k(V)^*$ of algebras, where $A_{k,g^{-1}}(V')$ is the higher level $g^{-1}$-twisted Zhu algebra of vertex operator algebra $V'$.

(ii) Suppose $\dim A_{k,g^{-1}}(V')<\infty$, there is a well-defined epimorphism $\Psi_g^k:C^k(V)\rightarrow A_{k,g^{-1}}(V')^*$ of coalgebras.

(iii) If $A_{k,g^{-1}}(V')$ is finite-dimensional, we have $\Phi_g^k,\Psi_g^k$ are isomorphisms.
\end{Proposition}
\proof  This is similar to section 6, we omit it.                                       $\hfill\Box$

\begin{Proposition}
Suppose $\dim C_g^k(V)<\infty$ for all $k\in\frac{1}{T}\mathbb{N}$.
Let $M$ be a $C_g^k(V)$-comodule which is not a $C^{k-\frac{1}{T}}(V)$-comodule. Then there is an admissible $g$-twisted $V$-comodule $\mathcal {L}_g^k(M)$ such that $\Omega_g^k/\Omega_g^{k-\frac{1}{T}}\circ\mathcal {L}_g^k(M)=M.$

Furthermore, $\mathcal {L}_g^k$ is a functor from the category of $C_g^k(V)$-comodules which are not $C_g^{k-\frac{1}{T}}(V)$-comodules to admissible $g$-twisted $V$-comodules category such that $\Omega_g^k/\Omega_g^{k-\frac{1}{T}}\circ\mathcal {L}_g^k=\Id$, it also sends simple objects to simple objects.
\end{Proposition}
\proof  This is similar to section 6, we omit it.                                       $\hfill\Box$

\begin{Theorem}\label{7.5}
If $V$ is a $g$-corational graded vertex operator coalgebra and $g\in\Aut V$ a finite order automorphism of order $T$ such that $\dim C_g^k(V)<\infty$, then $C_g^k(V)$ is a cosemisimple coassociative coalgebra.
\end{Theorem}
\proof  This is similar to section 6, we omit it.                                       $\hfill\Box$

\begin{Lemma}\cite{DLM3}\label{7.6}
 $\mathcal {V}$ is a $g^{-1}$-rational vertex operator algebra if and only if all its Zhu algebras $A_{k,g^{-1}}(\mathcal {V})$ are finite-dimensional semisimple associative algebras.
\end{Lemma}

\begin{Lemma}\cite{DNR}\label{7.7}
$C$ is a cosemisimple coassociative coalgebra if and only if its dual algebra $C^*$ is a semisimple associative algebra.
\end{Lemma}

\begin{Theorem}
Let $V$ be a graded vertex operator coalgebra and $g\in\Aut V$ a finite order automorphism of order $T$ such that $\dim A_{k,g^{-1}}(V')<\infty$ for all $k\in\frac{1}{T}\mathbb{N}$. Then $V$ is $g$-corational if and only if $V'$ is a $g^{-1}$-rational vertex operator algebra.
\end{Theorem}
\proof By Proposition \ref{Prop7.4}, Theorem \ref{7.5}, Lemma \ref{7.6} and \ref{7.7}, it is obvious that if $V$ is $g$-corational, then $V'$ is $g^{-1}$-rational.

Conversely, it is enough to show if all $C_g^k(V)$ are cosemisimple, then $V$ is $g$-corational. For any admissible $g$-twisted $V$-comodule $\mathcal {M}=\oplus_{i\in\mathbb{N}}M_i$, let $\soc(\mathcal {M})$ be the sum of all simple sub $g$-twisted $V$-comodules of $\mathcal {M}$. If $\soc(\mathcal {M})=\mathcal {M}$, we are done.

If not, there is $k\in\frac{1}{T}\mathbb{N}$ such that $\soc(\mathcal {M})_k\neq M_k$. Suppose $k$ is the smallest positive number such that $\soc(\mathcal {M})_k\neq M_k$. Since $C_g^k(V)$ is cosemisimple, we have $M_k$ is a cosemisimple $C_g^k(V)$-comodule, and $\soc(\mathcal {M})_k$ is a proper sub comodule of $M_k$. Thus there is at least one simple sub comodule $X_k$ of $M_k$ intersects $\soc(\mathcal {M})_k$ trivially. Suppose the irreducible admissible $g$-twisted $V$-comodule corresponding $X_k$ is $\mathcal {Z}$.

Let $\mathcal {X}$ be the admissible $g$-twisted sub $V$-comodule of $\mathcal {M}$ generated by $X_k$. If $\mathcal {X}$ is a simple admissible $g$-twisted $V$-comodule, then $\mathcal {X}\subseteq\soc(\mathcal {M})$, this is impossible. Hence $\mathcal {X}$ is reducible. Now $\mathcal {Z}$ is the unique irreducible quotient of $\mathcal {X}$. If there is $l$, such that $X_l\neq Z_l$. Since $X_l$ is a cosemisimple $C_g^l(V)$-comodule, there is a $C_g^l(V)$-comodule $S_l$ such that $X_l=Z_l\oplus S_l$. Let $\mathcal {S}$ be the admissible $g$-twisted sub $V$-comodule of $\mathcal {X}$ generated by $S_l$. Then we know $\mathcal {S}_k=0.$ Let $\mathcal {T}$ be an irreducible quotient of $\mathcal {S}$, then it is also an irreducible quotient of $\mathcal {X}$, which is different from $\mathcal {Z}$. This is a contradiction. Hence $\mathcal {X}=\mathcal {Z}$, and $\mathcal {X}$ is irreducible. This is also a contradiction, and this contradiction means $\soc(\mathcal {M})=\mathcal {M}$, i.e., $\mathcal {M}$ is completely reducible. Thus $V$ is $g$-corational.

This completes the proof.                                $\hfill\Box$

From \cite{DLM3}, we know that there is an anti-isomorphism between $A_{g,k}(\mathcal {V})$ and $A_{g^{-1},k}(\mathcal {V})$ for all $k\in\frac{1}{T}\mathbb{N}$, and an associative algebra $A$ is semisimple if and only if its opposite algebra $A^{op}$ is semisimple. Thus $\mathcal {V}$ is $g$-rational if and only if $\mathcal {V}$ is $g^{-1}$-rational \cite{DJ1}. Hence we get

\begin{Theorem}
Let $V$ be a graded vertex operator coalgebra and $g\in\Aut V$ a finite order automorphism of order $T$ such that $\dim A_{k,g}(V')<\infty$ for all $k\in\frac{1}{T}\mathbb{N}$. Then $V$ is $g$-corational if and only if $V'$ is a $g$-rational vertex operator algebra.
\end{Theorem}

\begin{Remark}
This is a generalization of a classical result, i.e., $A$ is a semisimple associative algebra if and only if its dual coassociative coalgebra $A^*$ is cosemisimple.
\end{Remark}

\end{document}